\documentclass[3p,times]{elsarticle}

\usepackage{amsmath}
\usepackage{color}
\usepackage{bm}
\usepackage[linesnumbered]{algorithm2e}
\usepackage{hyperref}

\usepackage{amssymb}
\usepackage{amsthm}

\DeclareMathOperator\diag{diag}
\DeclareMathOperator\spec{spec}

\biboptions{compress}

\newtheorem{prop}{Proposition}
\newtheorem{lemma}{Lemma}
\newtheorem{definition}{Definition}

\begin{document}

\begin{frontmatter}




\title{Lattice Boltzmann for linear elastodynamics: periodic problems and Dirichlet boundary conditions}


\author[1]{Oliver Boolakee}
\author[2]{Martin Geier}
\author[1]{Laura De Lorenzis\corref{cor1}}
\ead{ldelorenzis@ethz.ch}

\address[1]{Department of Mechanical and Process Engineering, ETH Zürich, 8092 Zürich, Switzerland}
\address[2]{Institute for Computational Modeling in Civil Engineering, TU Braunschweig, 38106 Braunschweig, Germany}
\cortext[cor1]{Corresponding author}

\begin{abstract}
We propose a new second-order accurate lattice Boltzmann formulation for linear elastodynamics that is stable for arbitrary combinations of material parameters under a CFL-like condition. The construction of the numerical scheme uses an equivalent first-order hyperbolic system of equations as an intermediate step, for which a vectorial lattice Boltzmann formulation is introduced. The only difference to conventional lattice Boltzmann formulations is the usage of vector-valued populations, so that all computational benefits of the algorithm are preserved. Using the asymptotic expansion technique and the notion of pre-stability structures we further establish second-order consistency as well as analytical stability estimates. Lastly, we introduce a second-order consistent initialization of the populations as well as a boundary formulation for Dirichlet boundary conditions on 2D rectangular domains. All theoretical derivations are numerically verified by convergence studies using manufactured solutions and long-term stability tests.
\end{abstract}

\begin{keyword}
Lattice Boltzmann method \sep Vectorial LBM \sep Elastodynamics.

\end{keyword}

\end{frontmatter}


\section{Introduction}

The lattice Boltzmann method (LBM) \cite{McNamara1988,He1997,Succi2001,Kruger2017} is designed to approximate solutions to a greatly simplified version of the Boltzmann equation known from gas dynamics. The primary solution quantity of the LBM are the so-called populations, which are obtained by evaluating the distribution function (the unknown of the Boltzmann equation) at a finite number of microscopic velocities times some weight. The gas kinetic origin that governs the evolution of these populations endows the derived numerical method with beneficial properties such as algorithmic simplicity and good scaling during parallelization \cite{Schoenherr2011,Pasquali2016,Kutscher2019}. The simplifications leading from the Boltzmann equation to the LBM are carefully chosen in such a way that the macroscopic behavior of the target partial differential equations (PDEs) is still retained. In the native variant of the LBM, the target PDEs are the incompressible Navier-Stokes equations, which are satisfied by the statistical moments of the populations with second-order accuracy \cite{Lallemand2021}. However, there exist many LBM formulations that solve other PDEs, including the diffusion  \cite{Wolf-Gladrow1995}, shallow water  \cite{Zhou2000} and wave equations \cite{Guangwu2000} and many more. Thereby, the LBM is considered purely from a numerical viewpoint as a discretization method to approximate solutions to a quite general class of PDEs.

More recently, a number of investigations taking this viewpoint attempted to benefit from the advantageous properties of the LBM to solve problems in solid mechanics, naturally starting from linear elasticity. For linear elastostatics, \cite{Yin2016} proposed a diffusion-type formulation in pseudo-time that converges at steady state to the solution of the elliptic linear elasticity equation. Based on this idea, we recently proposed a new second-order consistent LBM formulation for linear elastostatics that can handle arbitrary 2D geometries with Dirichlet and Neumann boundary conditions \cite{Boolakee2023,Boolakee2023a}. Despite retaining the favorable features of native LBM, this formulation is not competitive with existing numerical methods for linear elastostatics such as the finite element method, due to the need for time stepping to the steady-state solution. 

On the other hand, due to the inherent properties of LBM, it is reasonable to expect a competitive performance with respect to the finite element method for dynamic (and/or non-linear) problems. In this spirit, a number of works have been proposed to solve linear elastodynamics using LBM. Initial qualitative numerical experiments date back to \cite{Marconi2003,OBrien2012}, whereas \cite{Murthy2018} provides a more quantitative analysis establishing a new formulation. Based on this scheme a number of further developments for increased efficiency and boundary conditions are reported in \cite{Escande2020,Schluter2018,Schluter2021,Faust2024}. However, none of the previous works for elastodynamics has succeeded to establish second-order accuracy nor demonstrated stability for arbitrary combinations of material parameters.

In this paper, we seek to overcome the previous limitations, while remaining close to the basic algorithmic steps of the native LBM in order to retain all its computationally advantageous properties. Already in \cite{Murthy2018,Escande2020} it became apparent that LBM formulations based on a standard velocity set (i.\,e. populations with velocities that only travel to the next neighbors and not beyond) do not provide enough flexibility in order to independently adjust the different material parameters of linear elasticity while also accommodating the second time derivative that appears in the target equation. The approach proposed in \cite{Murthy2018,Escande2020,Faust2022,Faust2024} employs a standard velocity set along with the concept of moment chains, which needs to be combined with finite difference approximations in order to supply the missing quantities. However, the method requires significant artificial dissipation in order to remain stable on finite time intervals -- thus jeopardizing second-order consistency. Additionally, even with the added dissipation the von Neumann stability analysis in \cite{Escande2020} indicates unconditional instability except for one special combination of material parameters, for which the aforementioned finite difference correction cancels out.
One possible remedy involves using larger velocity sets \cite{Dubois2023}, but the numerical analysis turns out to be highly complex and the stability properties of those methods are still largely unclear. Additionally, the design of boundary conditions (especially for complex geometries) is expected to be very challenging and has -- to the best of the authors' knowledge -- never been addressed for arbitrary geometries. 

To circumvent the aforementioned issues, in this paper we adopt the so-called vectorial LBM, which involves vector-valued instead of scalar populations \cite{Dellar2009,Graille2014,Dubois2014,Zhao2020,Zhao2024,Bellotti2024}. In a sense, multi distribution formulations such as those in \cite{He1998m,Wang2013} can be also subsumed into the class of vectorial LBM. If we consider the so-called acoustic scaling, it is straightforward to show that the vectorial LBM with a vectorial counterpart of the well-known BGK collision operator approximates first-order hyperbolic systems of PDEs with second-order consistency in the non-dissipative limit \cite{Graille2014}. In fact, the vectorial LBM can be seen as a particularly simple numerical approximation of so-called relaxation schemes for systems of conservation laws \cite{Jin1995,Wissocq2023}.

Based on the vectorial LBM framework, we establish a novel formulation for linear elastodynamics with second-order consistency and rigorous stability estimates, that retains all computational benefits of conventional LBM and enables arbitrary combinations of material parameters. 
The manuscript is organized as follows: in Section \ref{sec:equation}, as a starting point, we summarize the target problem of linear elastodynamics, whereby we limit ourselves to periodic and Dirichlet problems, both on rectangular two-dimensional domains. We also introduce an equivalent reformulation of our target equation as a first-order hyperbolic system of PDEs and conclude with a nondimensionalization of the problem. Section \ref{sec:lbm} introduces the algorithmic structure of  vectorial LBM and describes all problem-dependent choices for our given target equation, initial and boundary conditions. The consistency of our proposed scheme is investigated in Section \ref{sec:analysis}, whereas the stability is the focus of Section \ref{sec:stability}. All analytical consistency and stability results are numerically verified in Section \ref{sec:verification}.

\section{Linear elastodynamics and its reformulation for LBM} \label{sec:equation}

In this section, we start by introducing our target problem, i.\,e. linear elastodynamics in 2D. We will show that it is possible to handle both a pure 2D setting \cite{Tokuoka1977} and 2D domains embedded in 3D space with the plane-strain or plane-stress assumption by modifying the definition of a single problem parameter. As a result, the entire analysis can be applied to all cases without any further changes. In this contribution, we restrict ourselves to unbounded domains with periodicity or Dirichlet problems on rectangular domains; an extension towards Neumann boundary conditions and more general problem geometries will be the focus of future work. Next, we introduce the equivalent reformulation of our target problem as a first-order hyperbolic system of equations as an intermediate step to construct the vectorial LBM formulation. Lastly, all quantities are nondimensionalized for the handling in LBM.

\subsection{Linear elastodynamics}

The target equation of linear elastodynamics reads
\begin{equation}
    \rho\partial_t^2 \boldsymbol{u} = \nabla \cdot \bar{\boldsymbol{\sigma}} + \bar{\boldsymbol{b}} \qquad \text{on }\Omega\times(0,t_f)\subset\mathbb{R}^2\times\mathbb{R},
\end{equation}
where $\boldsymbol{u}:\Omega\times(0,t_f)\rightarrow\mathbb{R}^2$ is the displacement, $\bar{\boldsymbol{\sigma}}$ the Cauchy stress tensor and $\bar{\boldsymbol{b}}$ the body load. We assume that all functions are sufficiently regular. The governing equation holds on the 2D domain $\Omega$ and until some final time $t_f>0$. The material density $\rho>0$ is assumed to be constant and can therefore be removed by re-scaling the Cauchy stress and the body load:
\begin{equation}
\label{eq:teq}
    \partial_t^2 \boldsymbol{u} = \nabla \cdot \boldsymbol{\sigma} + \boldsymbol{b} \qquad \text{on }\Omega\times(0,t_f)\subset\mathbb{R}^2\times\mathbb{R},
\end{equation}
with $\boldsymbol{b}:=\Bar{\boldsymbol{b}}/\rho$ and the linear elastic material law
\begin{equation}
\label{eq:cauchy}
    \boldsymbol{\sigma} := \bar{\boldsymbol{\sigma}}/\rho =  c_K^2 (\nabla \cdot \boldsymbol{u}) \boldsymbol{I} + c_{\mu}^2 \left(\nabla\boldsymbol{u}+(\nabla\boldsymbol{u})^T-(\nabla \cdot \boldsymbol{u}) \boldsymbol{I}\right).
\end{equation}
Here we introduced the elastic wave speeds $c_K$ and $c_{\mu}$ given by the following expressions
\begin{align}
\label{eq:2D}
    \text{2D continua:}\quad & c_K:=\sqrt{K/\rho}  \\
\label{eq:plane_strain}
    \text{Plane strain:}\quad & c_K:=\sqrt{\frac{3K+\mu}{3\rho}}  \\
    \label{eq:plane_stress}
    \text{Plane stress:}\quad & c_K:=\sqrt{\frac{9K\mu}{(3K+4\mu)\rho}}   \\
\label{eq:all}
    \text{All cases:}\quad & c_{\mu}:=\sqrt{\mu/\rho},
\end{align}
with $K,\mu>0$ as the bulk and shear modulus, respectively. 

The target equation is furnished with initial and boundary conditions, whereby we limit ourselves to unbounded domains with periodicity or the Dirichlet problem. In the former case we assume periodic behavior along each axis direction of the Cartesian reference frame, i.\,e.
\begin{equation}
    \label{eq:periodic}
    \boldsymbol{u}(\boldsymbol{x}+L_x\boldsymbol{e}_x+L_y\boldsymbol{e}_y,t) = \boldsymbol{u}(\boldsymbol{x},t) \qquad \forall (\boldsymbol{x},t)\in\mathbb{R}^2\times(0,t_f),
\end{equation}
where $L_x,L_y\in\mathbb{R}$ denote the periodic lengths along the two axis directions with orthonormal basis vectors $\boldsymbol{e}_x, \boldsymbol{e}_y\in\mathbb{R}^2$. In the latter case, we prescribe the following condition on the entire boundary
\begin{equation}
    \label{eq:bc}
    \boldsymbol{u} = \boldsymbol{u}_D \qquad \text{on }\partial\Omega\times(0,t_f),
\end{equation}
involving the boundary value function $\boldsymbol{u}_D$. In both cases, we additionally supply initial conditions for the displacement field as well as its first time derivative
\begin{alignat}{4}
\label{eq:init_u}
    \boldsymbol{u} &= \boldsymbol{u}_0 \qquad &&\text{on }\Omega\times\{0\} \\
    \label{eq:init_v}
    \partial_t\boldsymbol{u} &= \boldsymbol{v}_0 \qquad &&\text{on }\Omega\times\{0\},
\end{alignat}
where $\boldsymbol{u}_0$ and $\boldsymbol{v}_0$ denote the initial displacement and the initial velocity, respectively. Finally, we assume that the initial and the boundary conditions are compatible, so that smooth solutions are obtained.

\subsection{Reformulation as a first-order hyperbolic system}

As motivated in the introduction, we reformulate the target equation (Eq. \eqref{eq:teq}) as an equivalent first-order hyperbolic system of equations. To this end, let us first introduce a few intermediate quantities and express all equations with respect to a Cartesian reference frame with the already introduced basis vectors $\boldsymbol{e}_x$ and $\boldsymbol{e}_y$. Accordingly, Eq. \eqref{eq:teq} consists of two equations in the scalar components $u_x$ and $u_y$ of the displacement vector, i.\,e. $\boldsymbol{u} = u_x \boldsymbol{e}_x +  u_y \boldsymbol{e}_y$. Next, we define the velocity components as the time derivatives of the scalar displacement components
\begin{align}
\label{eq:vx}
    v_x &:= \partial_t u_x \\ v_y &:= \partial_t u_y, \label{eq:vy}
\end{align}
and further define the following expressions based on the spatial derivatives of the displacement field
\begin{align}
\label{eq:js}
    j_s &:= -c_K(\partial_x u_x + \partial_y u_y) \\
    \label{eq:jd}
    j_d &:= -c_{\mu}(\partial_x u_x - \partial_y u_y) \\
    \label{eq:jxy}
    j_{xy} &:= -c_{\mu}(\partial_x u_y + \partial_y u_x).
\end{align}
Hereby, $\partial_x$ and $\partial_y$ denote the space derivative along the x- and y-direction of the Cartesian basis, respectively. Using the definitions, Eq. \eqref{eq:teq} is transformed into an equivalent first-order hyperbolic system form
\begin{equation}
\label{eq:system}
\begin{split}
    &\partial_t \boldsymbol{U} + \partial_x \boldsymbol{\Phi}_x(\boldsymbol{U}) + \partial_y \boldsymbol{\Phi}_y(\boldsymbol{U}) = \boldsymbol{B} \qquad \text{on }\Omega\times(0,t_f) \\
    &\text{with } \boldsymbol{U}:= \left[\begin{matrix}
        v_x \\ v_y \\ j_s \\ j_d \\ j_{xy}
    \end{matrix}\right] \quad \boldsymbol{\Phi}_x(\boldsymbol{U}) := \left[\begin{matrix}
        c_K j_s + c_{\mu} j_d \\ c_{\mu} j_{xy} \\ c_K v_x \\ c_{\mu} v_x \\ c_{\mu} v_y
    \end{matrix}\right] \quad \boldsymbol{\Phi}_y(\boldsymbol{U}) := \left[\begin{matrix}
        c_{\mu} j_{xy} \\ c_K j_s - c_{\mu} j_d \\ c_K v_y \\ -c_{\mu} v_y \\ c_{\mu} v_x
    \end{matrix}\right] \quad \boldsymbol{B} := \left[\begin{matrix}
        b_x \\ b_y \\ 0 \\ 0 \\ 0
    \end{matrix}\right],
\end{split}
\end{equation}
where $b_x$ and $b_y$ denote the x- and y-component of the body load $\boldsymbol{b}$, respectively. 
It will turn out in Section \ref{sec:stability} that the specific form of the definitions in Eqs. \eqref{eq:js}-\eqref{eq:system} enables a straightforward stability analysis.

Furthermore, let us transform the initial and boundary conditions of Eqs. \eqref{eq:periodic}-\eqref{eq:init_v} into equivalent conditions for the system in Eq. \eqref{eq:system}. Utilizing the initial conditions for the displacement field and its time derivative in Eqs. \eqref{eq:init_u} and \eqref{eq:init_v}, we can directly generate the initial conditions for all quantities in the primary solution vector $\boldsymbol{U}$ by using the definitions of Eqs. \eqref{eq:vx}-\eqref{eq:jxy} as follows:
\begin{equation}
\label{eq:ic_vector}
    \boldsymbol{U}_0 = \left[\begin{matrix}
        v_{0x} & v_{0y} & -c_K(\partial_x u_{0x} +\partial_y u_{0y}) & -c_{\mu}(\partial_x u_{0x} -\partial_y u_{0y}) & -c_{\mu}(\partial_x u_{0y} +\partial_y u_{0x})
    \end{matrix}\right]^T \qquad \text{on } \Omega\times\{0\},
\end{equation}
where $u_{0x}$, $u_{0y}$, $v_{0x}$ and $v_{0y}$ respectively denote the x- and y-components of the initial displacement and of the initial velocity. Similarly, we transform the periodicity assumption of Eq. \eqref{eq:periodic} or the Dirichlet boundary condition in Eq. \eqref{eq:bc} into analogous expressions for the equivalent system. For the periodic domain, this yields
\begin{equation}
    \boldsymbol{U}(\boldsymbol{x}+L_x\boldsymbol{e}_x+L_y\boldsymbol{e}_y,t) = \boldsymbol{U}(\boldsymbol{x},t) \qquad \forall (\boldsymbol{x},t)\in\mathbb{R}^2\times(0,t_f).
\end{equation}
For the reformulation of the Dirichlet boundary condition, let us introduce the outward pointing unit normal vector on the boundary $\boldsymbol{n}=n_x\boldsymbol{e}_x+n_y\boldsymbol{e}_y$, with its components $n_x$ and $n_y$. Note that for the rectangular domains considered in this contribution, we have $(n_x,n_y)\in\{(1,0),(-1,0),(0,1),(0,-1)\}$. Altogether, we obtain a mixed-type boundary condition involving components of the primary solution vector $\boldsymbol{U}$ as well as its fluxes $\boldsymbol{\Phi}_x(\boldsymbol{U})$ and $\boldsymbol{\Phi}_y(\boldsymbol{U})$
\begin{equation}
\label{eq:bc_vector}
\begin{split}
    \boldsymbol{\Psi}(\boldsymbol{U}):=&\boldsymbol{I}_u\boldsymbol{U}+n_x\boldsymbol{I}_{\Phi}\boldsymbol{\Phi}_x(\boldsymbol{U}) + n_y\boldsymbol{I}_{\Phi}\boldsymbol{\Phi}_y(\boldsymbol{U}) = \boldsymbol{S}_{bc}\partial_t\boldsymbol{u}_D \quad \text{on } \partial\Omega\times(0,t_f) \\
    \text{with}\qquad &\begin{split}\boldsymbol{I}_u &:= \diag{(1,1,0,0,0)},\\\boldsymbol{I}_{\Phi} &:= \diag{(0,0,1,1,1)},\end{split} \qquad \boldsymbol{S}_{bc} :=\left[\begin{matrix}1&0\\0&1\\ c_Kn_x&c_Kn_y\\c_\mu n_x&-c_\mu n_y\\c_\mu ny&c_\mu n_x\end{matrix}\right],
\end{split}
\end{equation}
which involves the first time derivative of the Dirichlet boundary value $\boldsymbol{u}_D$.

\subsection{Nondimensionalization}

For the solution of the target equation using the LBM all expressions are transformed into dimensionless form, for which we introduce the reference length $L$ and time $T$. Because of the linearity of the problem we can scale the velocity solution field arbitrarily by some reference speed $V$. As a result, all dimensionless quantities involved in the system of Eq. \eqref{eq:system} are obtained as follows:
\begin{equation}
\label{eq:nondim}
    \left\{\begin{array}{l}
         \nabla =: L^{-1}\Tilde{\nabla} \\
         \partial_t =: T^{-1} \partial_{\Tilde{t}} \\
         \boldsymbol{v} =: V\Tilde{\boldsymbol{v}}
    \end{array}\right. \quad \Rightarrow \quad \boldsymbol{U} = V\Tilde{\boldsymbol{U}}, \qquad 
    \begin{matrix}
        \boldsymbol{\Phi}_x = LT^{-1}V\Tilde{\boldsymbol{\Phi}}_x \\
        \boldsymbol{\Phi}_y = LT^{-1}V\Tilde{\boldsymbol{\Phi}}_y
    \end{matrix}, \qquad 
    \boldsymbol{B} = T^{-1}V\Tilde{\boldsymbol{B}}, \qquad \begin{matrix}
        c_K = LT^{-1}\Tilde{c}_K \\ c_{\mu} = LT^{-1}\Tilde{c}_{\mu}
    \end{matrix}.
\end{equation}
With the dimensionless position $\boldsymbol{x}=L\Tilde{\boldsymbol{x}}$ and time $t=T\Tilde{t}$, we further define $t_f=T\Tilde{t}_f$ and the dimensionless domain
\begin{equation}
    \Tilde{\Omega} := \left\{\Tilde{\boldsymbol{x}}\in\mathbb{R}^2:L\Tilde{\boldsymbol{x}}=\boldsymbol{x}; \forall\boldsymbol{x}\in\Omega\right\},
\end{equation}
and similarly for the dimensionless domain boundary $\partial\Tilde{\Omega}$.

\section{Vectorial lattice Boltzmann formulation for linear elastodynamics} \label{sec:lbm}

As already shown in \cite{Graille2014,Dubois2014,Zhao2020,Zhao2024,Bellotti2024}, vectorial LBM is a promising generalization of the conventional LBM that allows for greater design freedom to solve coupled systems of PDEs while still retaining the computational advantages of the conventional LBM. Accordingly, we seek to find a second-order consistent vectorial LBM formulation to solve Eq. \eqref{eq:system}, which in turns yields a second-order consistent solution to the original target equation, i.\,e. Eq. \eqref{eq:teq}. In the following, we first introduce the lattice, time discretization and velocity set on which LBM operates as well as the key algorithmic steps of LBM using vector-valued populations. This is followed by the specific choices required to solve our target problem in Eq. \eqref{eq:system}, including formulations for the interior of the domain, initial conditions and Dirichlet boundary conditions. All these formulations are shown to be second-order consistent in Section \ref{sec:analysis}.

\subsection{Vectorial lattice Boltzmann algorithm}

As in conventional LBM, we discretize space and time with a uniform lattice involving the dimensionless grid spacing $\Delta x>0$ and the dimensionless time step size $\Delta t>0$. As a result, all quantities are evaluated only on the set of lattice node positions $\mathbb{G}$ and at discrete times $\mathbb{T}$, as defined below.
\begin{align}
    \mathbb{G}&:=\left\{\boldsymbol{x}\in\Tilde{\Omega}:\boldsymbol{x}=\boldsymbol{x}_0+(k\boldsymbol{e}_x+l\boldsymbol{e}_y)\Delta x;\;\forall (k,l)\in\mathbb{Z}^2\right\} \\
    \mathbb{T}&:=\left\{t\in(0,\Tilde{t}_f):t=m\Delta t;\;\forall m\in\mathbb{N}\right\}
\end{align}
where $\boldsymbol{x}_0\in\mathbb{R}^2$ accounts for some potentially nonzero offset of the lattice. 
From here on, all positions and times are by definition dimensionless if they are elements of $\mathbb{G}$ or $\mathbb{T}$. Therefore, we omit the $\Tilde{(\bullet)}$ when referring to space and time in order to avoid cluttering notation.

Additionally, we define a set of microscopic velocities
\begin{equation}
\label{eq:set_vel}
    \mathbb{V}:=\{\boldsymbol{c}_{ij}\in\mathbb{R}^2:\boldsymbol{c}_{ij}=c(i\boldsymbol{e}_x+j\boldsymbol{e}_y);\; \forall (i,j)\in\mathcal{V}\subset\mathbb{Z}^2\},
\end{equation}
with index set $\mathcal{V}$, for each of which we introduce a population $\boldsymbol{f}_{ij}$ that travels with the respective velocity $\boldsymbol{c}_{ij}$. Note that $c=\Delta x/\Delta t$, which by construction guarantees that all populations travel from one lattice node to another during each time step (neglecting lattice node positions next to possible domain boundaries).

\begin{figure}[htb]
    \centering
    \includegraphics{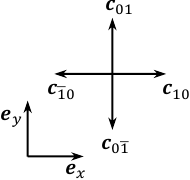}
    \caption{D2Q4 velocity set using Miller indices, i.\,e. $\bar{1}:=-1$.}
    \label{fig:d2q4}
\end{figure}

The novel aspect of vectorial LBM is that the populations are vector-valued functions, i.\,e. $\boldsymbol{f}_{ij}:\mathbb{G}\times\mathbb{T}\rightarrow\mathbb{R}^n$ for all $(i,j)\in\mathcal{V}$ and with nonzero $n\in\mathbb{N}$ as the number of scalar components. Each of the scalar components of the population set is used to approximate the solution of one of the scalar equations in Eq. \eqref{eq:system}, which by itself is a simple hyperbolic conservation law-type problem. The number of microscopic velocities, i.\,e. $|\mathbb{V}|$, determines the amount of information and the complexity of the problem that can be solved by the LBM formulation because it governs the number of independent moments that can be computed from the available populations \cite{Geier2015cum,Lallemand2021,Dubois2023}. In our case, each scalar equation in Eq. \eqref{eq:system} involves three quantities (one component of $\boldsymbol{U}$, $\boldsymbol{\Phi}_x$ and $\boldsymbol{\Phi}_y$), therefore we need at least three independent moments and velocities per scalar equation. In order to obtain a regular Cartesian lattice in 2D, we add one additional velocity and set the corresponding extra moment to zero (see Eq. \eqref{eq:cond_eq} in Section \ref{sec:lbm_elastodynamics}). Following the nomenclature in \cite{Graille2014}, we thus adopt a vectorial D2Q4$^5$ LBM stencil for our problem in Eq. \eqref{eq:system} with $n=5$ scalar equations. The corresponding index set of the velocities and populations is therefore defined as follows
\begin{equation}
\label{eq:index_vel}
    \mathcal{V} := \{(1,0),(0,1),(-1,0),(0,-1)\},
\end{equation}
see Figure \ref{fig:d2q4} for an illustration of the velocity set.
In order to compute the evolution of the vector-valued populations according to the LBM, we solve at each time step the lattice Boltzmann equation
\begin{equation}
\label{eq:lbe}
    \boldsymbol{f}_{ij}(\boldsymbol{x}+\boldsymbol{c}_{ij}\Delta t,t+\Delta t) = \omega \boldsymbol{f}_{ij}^{eq}\left(\boldsymbol{f}_{ij}(\boldsymbol{x},t);\Tilde{\boldsymbol{B}}(\boldsymbol{x},t)\right) + (1-\omega)\boldsymbol{f}_{ij}(\boldsymbol{x},t)+\Delta t(2-\omega)W_{ij}\Tilde{\boldsymbol{B}}(\boldsymbol{x},t) \qquad \forall(\boldsymbol{x},t)\in\mathbb{G}\times\mathbb{T}.
\end{equation}
The origin of this equation from gas dynamics and the main steps needed to obtain this numerical algorithm are summarized in \ref{app:lbm}. Conceptually, Eq. \eqref{eq:lbe} is decomposed into the two following stages:
\begin{alignat}{4}
\label{eq:collision}
    \text{Collision:}&\qquad \boldsymbol{f}_{ij}^*(\boldsymbol{x},t) = \omega \boldsymbol{f}_{ij}^{eq}\left(\boldsymbol{f}_{ij}(\boldsymbol{x},t);\Tilde{\boldsymbol{B}}(\boldsymbol{x},t)\right) + (1-\omega)\boldsymbol{f}_{ij}(\boldsymbol{x},t) + \Delta t(2-\omega)W_{ij}\Tilde{\boldsymbol{B}}(\boldsymbol{x},t) \qquad &&\forall(\boldsymbol{x},t)\in\mathbb{G}\times\mathbb{T}\\
    \label{eq:streaming}
    \text{Streaming:}&\qquad \boldsymbol{f}_{ij}(\boldsymbol{x}+\boldsymbol{c}_{ij}\Delta t,t+\Delta t) = \boldsymbol{f}_{ij}^*(\boldsymbol{x},t) \qquad &&\forall(\boldsymbol{x},t)\in\mathbb{G}\times\mathbb{T}.
\end{alignat}
Note that the streaming step written as above assumes an infinite domain and therefore neglects any special handling at the domain boundaries. Periodic and Dirichlet boundary conditions are discussed in Section \ref{sec:lbm_bcs_elastodynamics}.

The collision formulation in Eq. \eqref{eq:collision} represents the vectorial form of the well-known BGK collision \cite{BGK}, the simplest and most popular choice in LBM \cite{Lallemand2021}. With this choice, the collision stage is modeled by a relaxation towards a local equilibrium state $\boldsymbol{f}_{ij}^{eq}$, which itself depends on the current state of the populations as well as on some external forcing $\Tilde{\boldsymbol{B}}$. The dimensionless relaxation rate $\omega\in[0,2]$ governs "how fast" this relaxation towards equilibrium takes place, whereby $\omega\in(1,2]$ corresponds to the so-called overrelaxation \cite{Lallemand2021}. We will see in Section \ref{sec:analysis} that we require $\omega=2$ for second-order consistency, which eliminates the forcing term contribution in the collision rule of Eq. \eqref{eq:collision}. Thus, we leave the constants $W_{ij}$ undefined.  Note that the key design parameter of the collision operator is the definition of the equilibrium populations, which has to be designed in such as way as to approximate the solution of Eq. \eqref{eq:system} (see Section \ref{sec:lbm_elastodynamics}).

Subsequently, the post-collision populations $\boldsymbol{f}_{ij}^*$ are streamed to their neighbors according to Eq. \eqref{eq:streaming}. For all interior lattice nodes this corresponds to a simple shift operation in the implementation of the method. 

\subsection{Formulation for elastodynamics in the domain interior} \label{sec:lbm_elastodynamics}

The connection between the evolution of the populations and that of the solution quantities of Eq. \eqref{eq:system} is achieved through so-called moments. Due to the strong simplification of the distribution function being evaluated only at a few selected 
 velocities, the moment definitions simplify to weighted sums of the populations. As will be confirmed by the consistency analysis in Section \ref{sec:analysis}, we need to satisfy the following sufficient conditions for a consistent approximation of Eq. \eqref{eq:system}, i.\,e. $\lim_{\Delta t\rightarrow0,\Delta x\rightarrow0}\Tilde{\boldsymbol{U}}^{num}=\Tilde{\boldsymbol{U}}$:
\begin{equation}
\label{eq:cond_eq}
\left\{\begin{array}{rl}
     \Tilde{\boldsymbol{U}}^{num} &\overset{!}{=} \displaystyle\sum_{(i,j)\in\mathcal{V}} \boldsymbol{f}_{ij}^{eq}  \\
     \Tilde{{\boldsymbol{\Phi}}}_x(\Tilde{\boldsymbol{U}}^{num}) &\overset{!}{=} \displaystyle\sum_{(i,j)\in\mathcal{V}} ci\boldsymbol{f}_{ij}^{eq} \\
     \Tilde{{\boldsymbol{\Phi}}}_y(\Tilde{\boldsymbol{U}}^{num}) &\overset{!}{=} \displaystyle\sum_{(i,j)\in\mathcal{V}} cj\boldsymbol{f}_{ij}^{eq} \\
     \boldsymbol{0} &\overset{!}{=} \displaystyle\sum_{(i,j)\in\mathcal{V}} c^2(i^2-j^2)\boldsymbol{f}_{ij}^{eq}
\end{array}\right. \qquad\qquad \text{with }\Tilde{\boldsymbol{U}}^{num} := \sum_{(i,j)\in\mathcal{V}} \boldsymbol{f}_{ij} + \frac{\Delta t}{2}\Tilde{\boldsymbol{B}}
\end{equation}
Note that the definition of $\Tilde{\boldsymbol{U}}^{num}$ implies that we identify the numerical approximation of the primary solution vector with the so-called zeroth-order moment of the distribution function (plus some forcing to account for the body load).

The conditions for all equilibrium moments (Eq. \eqref{eq:cond_eq}) can be transformed into a single expression for the equilibrium populations such that all of them are simultaneously satisfied, leading to the following expression:
\begin{equation}
\label{eq:feq}
    \boldsymbol{f}_{ij}^{eq} = \frac{1}{4}\left[\Tilde{\boldsymbol{U}}^{num}+\frac{2}{c}\left(i\Tilde{{\boldsymbol{\Phi}}}_x(\Tilde{\boldsymbol{U}}^{num})+j\Tilde{{\boldsymbol{\Phi}}}_y(\Tilde{\boldsymbol{U}}^{num})\right)\right] \qquad\forall(i,j)\in\mathcal{V}.
\end{equation}

\subsection{Initial conditions for elastodynamics} \label{sec:lbm_ics_elastodynamics}
Based on the consistency analysis reported in Section \ref{sec:analysis}, we propose the following  initialization of the populations to prescribe the initial condition of Eq. \eqref{eq:ic_vector} with second-order accuracy:
\begin{align}
        \boldsymbol{f}_{ij}(\boldsymbol{x},0) &= \frac{1}{4}\left[\Tilde{\boldsymbol{U}}_0+\frac{2}{c}\left(i\Tilde{\boldsymbol{\Phi}}_x\left(\Tilde{\boldsymbol{U}}_0\right) +j\Tilde{\boldsymbol{\Phi}}_y\left(\Tilde{\boldsymbol{U}}_0\right)\right)\right] -\frac{\Delta t}{8}\left\{\left[\Tilde{\boldsymbol{B}}+\frac{2}{c}\left(i\Tilde{\boldsymbol{\Phi}}_x\left(\Tilde{\boldsymbol{B}}\right) +j\Tilde{\boldsymbol{\Phi}}_y\left(\Tilde{\boldsymbol{B}}\right)\right)\right]\right. \nonumber \\
        \label{eq:ic_second}
        &+c\left[ i\partial_x \Tilde{\boldsymbol{U}}_0 + \frac{2i^2-1}{c}\Tilde{\boldsymbol{\Phi}}_x\left(\partial_x\Tilde{\boldsymbol{U}}_0\right) - \frac{2}{c^2}\left(i\Tilde{\boldsymbol{\Phi}}_x\left(\Tilde{\boldsymbol{\Phi}}_x\left(\partial_x\Tilde{\boldsymbol{U}}_0\right)\right)+j\Tilde{\boldsymbol{\Phi}}_y\left(\Tilde{\boldsymbol{\Phi}}_x\left(\partial_x\Tilde{\boldsymbol{U}}_0\right)\right)\right)\right]\\ \nonumber
        &+c\left.\left[ j \partial_y \Tilde{\boldsymbol{U}}_0 + \frac{2j^2-1}{c}\Tilde{\boldsymbol{\Phi}}_y\left(\partial_y\Tilde{\boldsymbol{U}}_0\right) - \frac{2}{c^2}\left(i\Tilde{\boldsymbol{\Phi}}_x\left(\Tilde{\boldsymbol{\Phi}}_y\left(\partial_y\Tilde{\boldsymbol{U}}_0\right)\right)+j\Tilde{\boldsymbol{\Phi}}_y\left(\Tilde{\boldsymbol{\Phi}}_y\left(\partial_y\Tilde{\boldsymbol{U}}_0\right)\right)\right)\right]\right\} \qquad \forall \boldsymbol{x}\in\mathbb{G} \text{ and } \forall (i,j)\in\mathcal{V}.
    \end{align}
Note that on the right hand-side of Eq. \eqref{eq:ic_second} we omitted the argument $\boldsymbol{x}$ of $\Tilde{\boldsymbol{U}}_0$ for brevity. Further note that the initial condition requires up to second-order spatial derivatives of the displacement initial condition $\boldsymbol{u}_0$ and first-order spatial derivatives of the velocity initial condition $\boldsymbol{v}_0$, which we assume to be computable from the available data.

\subsection{Boundary conditions for elastodynamics} \label{sec:lbm_bcs_elastodynamics}
\subsubsection{Periodic boundary conditions}
\label{sec:pbc}
In the case of rectangular periodic domains, 
as illustrated in Figure \ref{fig:periodic}, the populations that are streamed out of the domain on the right side are applied as incoming populations on the left side of the domain, and vice versa for outgoing populations on the left side and incoming ones on the right side. Thus, the two sides are coupled together by the periodicity assumption. A similar procedure is performed at the top and bottom side, therefore obtaining a rectangular periodic domain.
\begin{figure}[htb]
    \centering
    \includegraphics{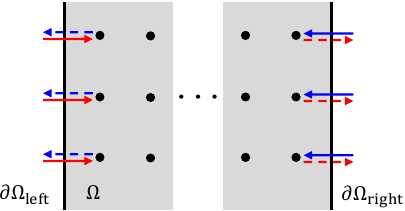}
    \caption[Periodicity along the horizontal direction.]{Periodicity along the horizontal direction. The outgoing populations on the right side (red dashed arrows) are streamed back into the domain on the left side (red solid arrows), and vice versa from left to right (blue arrows).}
    \label{fig:periodic}
\end{figure}
\subsubsection{Dirichlet boundary conditions}
\label{sec:dbc}
Next, we illustrate a boundary formulation for the Dirichlet boundary condition introduced in Eq. \eqref{eq:bc_vector}, which will be shown to be second-order consistent in Section \ref{sec:analysis}. 
Preliminarily, we introduce a few useful definitions, see also the visualization in Figure \ref{fig:boundary}.
First, we define the set of boundary nodes, which are all lattice nodes with missing neighbors because they lie outside the problem domain:
\begin{equation}
    \mathbb{B} := \left\{\boldsymbol{x}\in\mathbb{G}:\exists(i,j)\in\mathcal{V};\; \boldsymbol{x}-\boldsymbol{c}_{ij}\Delta t \notin \mathbb{G}\right\}.
\end{equation}
We further introduce an index set $\mathcal{D}_{\boldsymbol{x}}\subset\mathcal{V}$, which contains the velocity indices of all missing incoming populations for a given lattice node $\boldsymbol{x}\in\mathbb{G}$:
\begin{equation}
    \mathcal{D}_{\boldsymbol{x}} := \left\{(i,j)\in\mathcal{V}:\boldsymbol{x}-\boldsymbol{c}_{ij}\Delta t \notin \mathbb{G} \text{ for a given }\boldsymbol{x}\in\mathbb{G}\right\}.
\end{equation}
Note that for all interior nodes this set is empty, i.\,e. $\mathcal{D}_{\boldsymbol{x}}=\emptyset$ for all $\boldsymbol{x}\in\mathbb{G}\setminus\mathbb{B}$. For convenience, we further introduce the negation of an index set as follows:
\begin{equation}
    -\mathcal{V}:=\left\{(\Bar{\imath},\Bar{\jmath})\in\mathcal{V}:\forall(i,j)\in\mathcal{V}\right\} \qquad \text{with } \Bar{\imath}:=-i.
\end{equation}
If we apply this to $\mathcal{D}_{\boldsymbol{x}}$ for a given boundary node $\boldsymbol{x}\in\mathbb{B}$, $-\mathcal{D}_{\boldsymbol{x}}\subset\mathcal{V}$ is the index set of populations of this node that would travel outside the domain under normal streaming. Finally, we define the normalized wall distance $q_{ij}\in(0,1]$
\begin{equation}
    q_{ij} := \frac{|\boldsymbol{x}-\boldsymbol{x}_b|}{|\boldsymbol{c}_{ij}\Delta t|} \qquad \forall \boldsymbol{x}\in\mathbb{B} \text{ and } \forall (i,j)\in\mathcal{D}_{\boldsymbol{x}},
\end{equation}
where $\boldsymbol{x}_b = \boldsymbol{x}_b\left(\boldsymbol{x},(i,j)\right)$ is the intersection point of the lattice link with the physical boundary $\partial\Omega$ for a given boundary node $\boldsymbol{x}\in\mathbb{B}$ and velocity index $(i,j)\in\mathcal{D}_{\boldsymbol{x}}$. In summary, our boundary formulation needs to provide suitable values for all missing incoming populations $\boldsymbol{f}_{ij}$ at each boundary node and lattice link which intersects the boundary.

\begin{figure}[htb]
    \centering
    \includegraphics{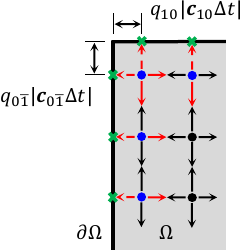}
    \caption[Visualization of node and index sets at the domain boundary.]{Visualization of node and index sets at the domain boundary. Interior nodes $\mathbb{G}\setminus\mathbb{B}$ -- black dots; boundary nodes $\mathbb{B}$ -- blue dots; normal streaming indices $\mathcal{V}\setminus\mathcal{D}_{\boldsymbol{x}}$ for each $\boldsymbol{x}\in\mathbb{G}$ --  black arrows; missing incoming populations $\mathcal{D}_{\boldsymbol{x}}$ for each $\boldsymbol{x}\in\mathbb{G}$ -- red solid arrows; outgoing populations $-\mathcal{D}_{\boldsymbol{x}}$ for each $\boldsymbol{x}\in\mathbb{G}$ -- red dashed arrows; boundary intersection points $\boldsymbol{x}_b$ -- green crosses.}
    \label{fig:boundary}
\end{figure}

With the above definitions, we propose the following second-order consistent Dirichlet boundary formulation:
\begin{alignat}{3}
    \label{eq:bc_ansatz}
    \boldsymbol{f}_{ij}(\boldsymbol{x},t+\Delta t) &= \boldsymbol{D} \boldsymbol{f}^*_{\Bar{\imath}\Bar{\jmath}}(\boldsymbol{x},t) + \boldsymbol{S}_{ij}(\boldsymbol{x}_b,t+\Delta t/2) \qquad &&\forall (\boldsymbol{x},t)\in\mathbb{B}\times\mathbb{T} \text{ and } \forall (i,j)\in\mathcal{D}_{\boldsymbol{x}} \\ \nonumber &&\text{ and } &\forall (\boldsymbol{x}_b,t+\Delta t/2)\in\partial\Tilde{\Omega}\times(0,\Tilde{t}_f); \\
    \text{with}\qquad \boldsymbol{D} &:= \diag(-1,-1,1,1,1)&& \label{eq:bc_d} \\
    \boldsymbol{S}_{ij} &:= \left[\begin{array}{rr}
    1/2 & 0 \\ 0 & 1/2 \\ i\Tilde{c}_K & j \Tilde{c}_K \\ i\Tilde{c}_{\mu} & -j\Tilde{c}_{\mu} \\ j\Tilde{c}_{\mu} & i\Tilde{c}_{\mu}
    \end{array}\right]\partial_t \Tilde{\boldsymbol{u}}_D && \label{eq:bc_s} \\
    q_{ij} &= 1/2  &&\forall \boldsymbol{x}\in\mathbb{B} \text{ and } \forall (i,j)\in\mathcal{D}_{\boldsymbol{x}} \label{eq:bc_q}
\end{alignat}

Note that as a consequence of the definition of $\boldsymbol{D}$ in Eq. \eqref{eq:bc_d}, we employ the anti bounce-back rule for the first two components and the bounce-back rule for the remainder of the population vector. This enables us to fully define a consistent reconstruction of the missing incoming population based only on the time derivative of the Dirichlet boundary condition, where Eq. \eqref{eq:bc_s} can be seen as the discretized counterpart of $\boldsymbol{S}_{bc}$ in Eq. \eqref{eq:bc_vector}. The strong restriction given by Eq. \eqref{eq:bc_q} implies that - so far - we can only analyze Dirichlet problems on rectangular domains, where the problem domain boundary lies exactly half a lattice spacing away from the boundary node.

\subsection{Postprocessing} \label{sec:lbm_postprocessing}
Because we are numerically approximating the equivalent problem of Eq. \eqref{eq:system}, we do not directly obtain a numerical approximation of the displacement solution $\boldsymbol{u}^{num}$. However, using the velocity approximation,
\begin{equation}
    \boldsymbol{v}^{num} := V\left[\begin{matrix}1&0&0&0&0\\0&1&0&0&0\end{matrix}\right]\Tilde{\boldsymbol{U}}^{num},
\end{equation}
we can compute the displacement solution with second-order consistency using the trapezoidal rule:
\begin{equation}
\label{eq:trapezoidal}
    \begin{array}{ll}
        \boldsymbol{u}^{num}(\boldsymbol{x},t) = \boldsymbol{u}^{num}(\boldsymbol{x},t-\Delta t) + \frac{\Delta t}{2}\left(\boldsymbol{v}^{num}(\boldsymbol{x},t-\Delta t) + \boldsymbol{v}^{num}(\boldsymbol{x},t)\right) + \mathcal{O}\left(\Delta t^2\right) &\qquad\forall(\boldsymbol{x},t)\in\mathbb{G}\times\mathbb{T}; \\
        \boldsymbol{u}^{num}(\boldsymbol{x},0) = \boldsymbol{u}_0(\boldsymbol{x}) &\qquad\forall\boldsymbol{x}\in\mathbb{G}.
    \end{array}
\end{equation}
In order to avoid storing the previous velocity solution $\boldsymbol{v}^{num}(\boldsymbol{x},t-\Delta t)$, we already add its contribution according to Eq. \eqref{eq:trapezoidal} at the end of the previous iteration and use this intermediate quantity $\boldsymbol{u}^*$ to obtain the displacement solution at the next time step (see Section \ref{sec:summary}). Therefore, we require only the current velocity solution $\boldsymbol{v}^{num}(\boldsymbol{x},t)$ to compute the displacement approximation $\boldsymbol{u}^{num}(\boldsymbol{x},t)$. Note that second-order accuracy of the displacement solution critically relies on the second-order consistency of the lattice Boltzmann approximation of $\boldsymbol{U}$ to be demonstrated in Section \ref{sec:analysis}.

Additionally, the numerical approximation of the Cauchy stress components is directly obtained from the primary solution vector $\Tilde{\boldsymbol{U}}^{num}$ as follows:
\begin{equation}
    \label{eq:sigma_num}
    \boldsymbol{\sigma}^{num}:=\left[\begin{matrix}\sigma_{xx}^{num}\\\sigma_{yy}^{num}\\\sigma_{xy}^{num}
    \end{matrix}\right] = -V\left[\begin{matrix}0&0&c_K&c_\mu&0\\0&0&c_K&-c_\mu&0\\0&0&0&0&c_\mu\end{matrix}\right]\Tilde{\boldsymbol{U}}^{num}.
\end{equation}

\subsection{Summary} \label{sec:summary}
In summary, the following steps are necessary for a single time step of our lattice Boltzmann formulation for linear elastodynamics:
\begin{enumerate}
    \item Collision
    \begin{enumerate}[(a)]
        \item Compute the vector of the primary solution variables, i.\,e. $\Tilde{\boldsymbol{U}}^{num}=\sum_{(i,j)\in\mathcal{V}}\boldsymbol{f}_{ij}+\frac{\Delta t}{2}\Tilde{\boldsymbol{B}}$.
        \item Compute the displacement solution, i.\,e. $\boldsymbol{u}^{num} = \boldsymbol{u}^* + \frac{\Delta t}{2}\boldsymbol{v}^{num}$, and the Cauchy stress solution using Eq. \eqref{eq:sigma_num}.
        \item Construct the local equilibrium populations using Eq. \eqref{eq:feq}.
        \item Perform the collision step (Eq. \eqref{eq:collision}) to obtain the post-collision populations $\boldsymbol{f}_{ij}^*$.
    \end{enumerate}
    \item Streaming
    \begin{enumerate}[(a)]
        \item Execute streaming in the domain interior (see Eq. \eqref{eq:streaming}).
        \item Perform periodic boundary handling (Section \ref{sec:pbc}) or apply Dirichlet boundary conditions (Section \ref{sec:dbc}).
        \item Prepare the displacement solution, i.\,e. $\boldsymbol{u}^*=\boldsymbol{u}^{num} + \frac{\Delta t}{2}\boldsymbol{v}^{num}$, using the same $\boldsymbol{v}^{num}$ computed during collision.
    \end{enumerate}
\end{enumerate}

\section{Consistency of the proposed method} \label{sec:analysis}

In this section, using the asymptotic expansion technique \cite{Junk2005,Geier2015cum}, we analyze the consistency of the formulation proposed in Section \ref{sec:lbm} using the D2Q4$^5$ velocity set, i.\,e. the velocity set defined in Eq. \eqref{eq:set_vel} with indices as in Eq. \eqref{eq:index_vel}. For the formulation in the interior of the domain, in order to achieve second-order consistency, 
we require zero dissipation during collision, i.\,e. $\omega=2$. 
We also prove second-order consistency for the proposed initial condition and Dirichlet boundary conditions. Also for boundary conditions, the analysis is based on the asymptotic expansion technique in a way similar to \cite{Junk2005BC,Yang2007,Boolakee2023a}.

\subsection{Second-order consistency in the domain interior}

First, we investigate the second-order consistency of the formulation introduced in Section \ref{sec:lbm_elastodynamics}. During the proof using asymptotic analysis it will turn out that the second-order consistency in the domain interior critically depends on the accuracy of the initialization (see Section \ref{sec:lbm_ics_elastodynamics}) and the boundary formulation (see Section \ref{sec:dbc}). To this end, we define the notion of a \emph{second-order consistent initialization} and a \emph{second-order consistent boundary formulation}, which implies consistent behavior with respect to the physical initial or boundary condition at leading order of the asymptotic expansion and zero initial or boundary conditions at the next higher order. Accordingly, let us introduce all requirements in the following proposition:

\begin{prop}[Consistency in the domain interior] \label{prop:consistency_interior}
    The vectorial lattice Boltzmann algorithm (Eq. \eqref{eq:lbe}) on a periodic domain is second-order consistent w.\,r.\,t. Eq. \eqref{eq:system} if the following conditions are satisfied:
    \begin{enumerate}[(i)]
        \item Definition of $\Tilde{\boldsymbol{U}}^{num}$ as in Eq. \eqref{eq:cond_eq};
        \item Equilibrium populations given by Eq. \eqref{eq:feq}, therefore satisfying Eq. \eqref{eq:cond_eq};
        \item Acoustic scaling, i.\,e. $c=\Delta x/\Delta t=const$;
        \item $\omega=2$;
        \item Second-order consistent initialization and second-order order consistent boundary formulation.
    \end{enumerate}
\end{prop}
\begin{proof}
    We assume the following formal polynomial expansions to hold for $\boldsymbol{f}_{ij}$, $\boldsymbol{f}_{ij}^{eq}$, $\Tilde{\boldsymbol{U}}^{num}$, $\Tilde{\boldsymbol{\Phi}}_x$ and $\Tilde{\boldsymbol{\Phi}}_y$:
    \begin{equation}
    \label{eq:formal}
    \begin{split}
        &\boldsymbol{f}_{ij}=\sum_{k=0}^\infty \Delta t^k \boldsymbol{f}_{ij}^{(k)}, \qquad \boldsymbol{f}_{ij}^{eq}=\sum_{k=0}^\infty \Delta t^k \boldsymbol{f}_{ij}^{(k)eq}\qquad\forall (i,j)\in\mathcal{V} \\
        &\Tilde{\boldsymbol{U}}^{num} = \sum_{k=0}^\infty \Delta t^k\boldsymbol{U}^{(k)},\qquad\Tilde{\boldsymbol{\Phi}}_x(\Tilde{\boldsymbol{U}}^{num}) = \sum_{k=0}^\infty \Delta t^k\Tilde{\boldsymbol{\Phi}}_x\left(\boldsymbol{U}^{(k)}\right),\qquad\Tilde{\boldsymbol{\Phi}}_y(\Tilde{\boldsymbol{U}}^{num}) = \sum_{k=0}^\infty \Delta t^k\Tilde{\boldsymbol{\Phi}}_y\left(\boldsymbol{U}^{(k)}\right),
    \end{split}
    \end{equation}
    where $\boldsymbol{f}_{ij}^{(k)}:\Tilde{\Omega}\times(0,\Tilde{t}_f)\rightarrow\mathbb{R}^5$ for all $k\in\mathbb{N}$. Additionally, we assume that each $\boldsymbol{f}_{ij}^{(k)}$ is sufficiently differentiable in order to carry out the asymptotic expansion up to the desired order (for details refer to \cite{Junk2004,Junk2005,Yang2007}). More precisely, the differentiability requirements only apply to the leading orders (in our case $k=0,1$), whereas higher-order components simply need to be bounded, which will be ensured by the stability established in Proposition \ref{prop:stab_periodic}. By the relations in Eq. \eqref{eq:cond_eq}, the same applies to all other expansion coefficient functions in Eq. \eqref{eq:formal}, where we additionally exploited the linearity of $\Tilde{\boldsymbol{\Phi}}_x$ and $\Tilde{\boldsymbol{\Phi}}_y$.
    
    Next, we introduce the formal expansion into the lattice Boltzmann equation, i.\,e. Eq. \eqref{eq:lbe}, and Taylor expand the coefficient functions around some arbitrary $(\boldsymbol{x},t)\in\Tilde{\Omega}\times(0,\Tilde{t}_f)$. In this step, we further make use of the acoustic scaling assumption (iii), i.\,e. $c=\Delta x/\Delta t = const$. After sorting by powers of $\Delta t$, we obtain in the first three orders:
    \begin{alignat}{4}
    \label{eq:order0}
        &\Delta t^0:\quad&\boldsymbol{0}&=\omega\left(\boldsymbol{f}_{ij}^{(0)eq}-\boldsymbol{f}_{ij}^{(0)}\right) \\
        \label{eq:order1}
        &\Delta t^1:\quad&\partial_t\boldsymbol{f}_{ij}^{(0)} + \boldsymbol{c}_{ij}\cdot\nabla\boldsymbol{f}_{ij}^{(0)}&=\omega\left(\boldsymbol{f}_{ij}^{(1)eq}-\boldsymbol{f}_{ij}^{(1)}\right) + (2-\omega)W_{ij}\Tilde{\boldsymbol{B}} \\
        \label{eq:order2}
        &\Delta t^2:\quad&\partial_t\boldsymbol{f}_{ij}^{(1)} + \boldsymbol{c}_{ij}\cdot\nabla\boldsymbol{f}_{ij}^{(1)}+\frac{1}{2}\partial^2_t\boldsymbol{f}_{ij}^{(0)} + \boldsymbol{c}_{ij}\cdot\partial_t\nabla\boldsymbol{f}_{ij}^{(0)}+ \frac{1}{2}\boldsymbol{c}_{ij}\otimes\boldsymbol{c}_{ij}\cdot\nabla\otimes\nabla\boldsymbol{f}_{ij}^{(0)}&=\omega\left(\boldsymbol{f}_{ij}^{(2)eq}-\boldsymbol{f}_{ij}^{(2)}\right),
    \end{alignat}
    where we omitted the argument $(\boldsymbol{x},t)$ for all functions. From Eqs. \eqref{eq:order0} and \eqref{eq:order1} we obtain
    \begin{align}
    \label{eq:order0_solved}
        \boldsymbol{f}_{ij}^{(0)}&=\boldsymbol{f}_{ij}^{(0)eq} \\
        \label{eq:order1_solved}
        \boldsymbol{f}_{ij}^{(1)}&=\boldsymbol{f}_{ij}^{(1)eq}-\frac{1}{2}\left( \partial_t\boldsymbol{f}_{ij}^{(0)eq} + \boldsymbol{c}_{ij}\cdot\nabla\boldsymbol{f}_{ij}^{(0)eq}\right),
    \end{align}
    where we applied (iv), i.\,e. $\omega=2$. If we substitute Eqs. \eqref{eq:order0_solved} and \eqref{eq:order1_solved} into Eq. \eqref{eq:order2} and with $\omega=2$, we obtain
    \begin{equation}
    \label{eq:order2_solved}
        \partial_t\boldsymbol{f}_{ij}^{(1)eq} + \boldsymbol{c}_{ij}\cdot\nabla\boldsymbol{f}_{ij}^{(1)eq} =2\left(\boldsymbol{f}_{ij}^{(2)eq}-\boldsymbol{f}_{ij}^{(2)}\right).
    \end{equation}
    Finally, taking zeroth-order moments of Eq. \eqref{eq:order1} (considering Eq.\eqref{eq:order0_solved}) and Eq. \eqref{eq:order2_solved}, i.\,e. summing over all populations, yields the following results:
    \begin{alignat}{4}
    \label{eq:eqpde1}
        &\Delta t^1:\quad&\partial_t\boldsymbol{U}^{(0)}+\partial_x \Tilde{\boldsymbol{\Phi}}_x\left(\boldsymbol{U}^{(0)}\right)+\partial_y \Tilde{\boldsymbol{\Phi}}_y\left(\boldsymbol{U}^{(0)}\right) &=\Tilde{\boldsymbol{B}} \\
        \label{eq:eqpde2}
        &\Delta t^2:\quad&\partial_t\boldsymbol{U}^{(1)}+\partial_x \Tilde{\boldsymbol{\Phi}}_x\left(\boldsymbol{U}^{(1)}\right) + \partial_y \Tilde{\boldsymbol{\Phi}}_y\left(\boldsymbol{U}^{(1)}\right) &=\boldsymbol{0},
    \end{alignat}
    where we used (ii) and (iv). Also note that, due to (i), it is $\displaystyle\sum_{(i,j)\in\mathcal{V}} \left(\boldsymbol{f}_{ij}^{(1)eq}-\boldsymbol{f}_{ij}^{(1)}\right)=\Tilde{\boldsymbol{B}}/2$.
    
    In summary, we see from Eq. \eqref{eq:eqpde1} that the leading-order expansion coefficient of the numerical approximation of the primary solution vector $\boldsymbol{U}^{(0)}$ indeed satisfies Eq. \eqref{eq:system}. If we combine this with (v), by which we assume consistent initial and boundary conditions at leading order, we obtain a formulation that is consistent with the target problem. Additionally, the leading-order error contribution $\boldsymbol{U}^{(1)}$ is the solution of the homogeneous target equation, which admits a zero solution if we supply zero initial and boundary conditions at this order. Again, this is provided by (v), which lets us conclude second-order consistency in the domain interior.
\end{proof}
\subsection{Second-order consistent initialization} \label{sec:ic}
 
As follows, we show that for the lattice Boltzmann scheme satisfying (i)-(iv) of Proposition \ref{prop:consistency_interior}, we obtain a second-order consistent initialization (in the sense that $\boldsymbol{U}^{(0)}=\Tilde{\boldsymbol{U}}_0$ and $\boldsymbol{U}^{(1)}=\boldsymbol{0})$ if we set $\boldsymbol{f}_{ij}$ as in Eq. \eqref{eq:ic_second}.
    Let us first express the first two leading orders of the formal asymptotic expansion of the populations explicitly in terms of the primary solution quantity. To this end, we start with Eqs. \eqref{eq:order0_solved} and \eqref{eq:order1_solved} and introduce the equilibrium definition of Eq. \eqref{eq:feq}. This leads to the following results for the expansion coefficients of $\boldsymbol{f}_{ij}$:
    \begin{alignat}{4}
        \label{eq:asymp_ic0}
        &\Delta t^0:\quad\boldsymbol{f}_{ij}^{(0)} = &&\frac{1}{4}\left[\boldsymbol{U}^{(0)}+\frac{2}{c}\left(i\Tilde{\boldsymbol{\Phi}}_x\left(\boldsymbol{U}^{(0)}\right) +j\Tilde{\boldsymbol{\Phi}}_y\left(\boldsymbol{U}^{(0)}\right)\right)\right] \\
        \label{eq:asymp_ic1}
        &\Delta t^1:\quad \boldsymbol{f}_{ij}^{(1)} = &&\frac{1}{4}\left[\boldsymbol{U}^{(1)}+\frac{2}{c}\left(i\Tilde{\boldsymbol{\Phi}}_x\left(\boldsymbol{U}^{(1)}\right) +j\Tilde{\boldsymbol{\Phi}}_y\left(\boldsymbol{U}^{(1)}\right)\right)\right] -\frac{1}{8}\left[\partial_t\boldsymbol{U}^{(0)}+\frac{2}{c}\left(i\Tilde{\boldsymbol{\Phi}}_x\left(\partial_t\boldsymbol{U}^{(0)}\right) +j\Tilde{\boldsymbol{\Phi}}_y\left(\partial_t\boldsymbol{U}^{(0)}\right)\right)\right] \\\nonumber
        &&& -\frac{c}{8} \left[i\partial_x \boldsymbol{U}^{(0)}+\frac{2}{c}i^2\Tilde{\boldsymbol{\Phi}}_x\left(\partial_x\boldsymbol{U}^{(0)}\right) + j\partial_y \boldsymbol{U}^{(0)}+\frac{2}{c}j^2\Tilde{\boldsymbol{\Phi}}_y\left(\partial_y\boldsymbol{U}^{(0)}\right) \right].
    \end{alignat}
    Note that $\partial_t \boldsymbol{U}^{(0)}$ involves up to second-order time derivatives of the displacement solution, for which no initial condition exists. To overcome this issue, we apply the target equation, i.\,e. Eq. \eqref{eq:eqpde1}, to turn the time derivative in Eq. \eqref{eq:asymp_ic1} into known (or computable) space derivatives of the displacement and velocity initial conditions, respectively:
    \begin{equation}
    \label{eq:asymp_ic1_mod}
    \begin{split}
        \boldsymbol{f}_{ij}^{(1)} = &\frac{1}{4}\left[\boldsymbol{U}^{(1)}+\frac{2}{c}\left(i\Tilde{\boldsymbol{\Phi}}_x\left(\boldsymbol{U}^{(1)}\right) +j\Tilde{\boldsymbol{\Phi}}_y\left(\boldsymbol{U}^{(1)}\right)\right)\right] -\frac{1}{8}\left[\Tilde{\boldsymbol{B}}+\frac{2}{c}\left(i\Tilde{\boldsymbol{\Phi}}_x\left(\Tilde{\boldsymbol{B}}\right) +j\Tilde{\boldsymbol{\Phi}}_y\left(\Tilde{\boldsymbol{B}}\right)\right)\right] \\
        &-\frac{c}{8}\left[ i\partial_x \boldsymbol{U}^{(0)} + \frac{2i^2-1}{c}\Tilde{\boldsymbol{\Phi}}_x\left(\partial_x\boldsymbol{U}^{(0)}\right) - \frac{2}{c^2}\left(i\Tilde{\boldsymbol{\Phi}}_x\left(\Tilde{\boldsymbol{\Phi}}_x\left(\partial_x\boldsymbol{U}^{(0)}\right)\right)+j\Tilde{\boldsymbol{\Phi}}_y\left(\Tilde{\boldsymbol{\Phi}}_x\left(\partial_x\boldsymbol{U}^{(0)}\right)\right)\right)\right]\\
        &-\frac{c}{8}\left[ j \partial_y \boldsymbol{U}^{(0)} + \frac{2j^2-1}{c}\Tilde{\boldsymbol{\Phi}}_y\left(\partial_y\boldsymbol{U}^{(0)}\right) - \frac{2}{c^2}\left(i\Tilde{\boldsymbol{\Phi}}_x\left(\Tilde{\boldsymbol{\Phi}}_y\left(\partial_y\boldsymbol{U}^{(0)}\right)\right)+j\Tilde{\boldsymbol{\Phi}}_y\left(\Tilde{\boldsymbol{\Phi}}_y\left(\partial_y\boldsymbol{U}^{(0)}\right)\right)\right)\right].
    \end{split}
    \end{equation}
    Finally, we can combine the results of Eqs. \eqref{eq:asymp_ic0} and \eqref{eq:asymp_ic1_mod}, i.\,e. $\boldsymbol{f}_{ij}(\boldsymbol{x},0) = \boldsymbol{f}_{ij}^{(0)}(\boldsymbol{x},0)+\Delta t \boldsymbol{f}_{ij}^{(1)}(\boldsymbol{x},0) + \mathcal{O}\left(\Delta t^2\right)$ and introduce the expansion into the left-hand side of Eq. \eqref{eq:ic_second}. For arbitrary $\Delta t>0$, the equation is satisfied if and only if $\boldsymbol{U}^{(0)}(\boldsymbol{x},0)=\Tilde{\boldsymbol{U}}_0(\boldsymbol{x})$ and $\boldsymbol{U}^{(1)}(\boldsymbol{x},0)=\boldsymbol{0}$ for all $\boldsymbol{x}\in\mathbb{G}$.

\subsection{Second-order consistent boundary formulation} \label{sec:bc_dirichlet}
In the presence of the Dirichlet boundary condition in Eq. \eqref{eq:bc_vector}, in order to obtain a second-order consistent boundary formulation, we require a consistent behavior of the boundary formulation at leading order, while prescribing a homogeneous boundary condition at the next higher order.

\begin{prop}[Second-order consistent Dirichlet boundary formulation] \label{prop:consistency_dbc}
    For the lattice Boltzmann scheme satisfying (i)-(iv) of Proposition \ref{prop:consistency_interior}, the  Dirichlet boundary formulation in Eqs. \eqref{eq:bc_ansatz}-\eqref{eq:bc_q} is second-order consistent, in the sense that $\boldsymbol{\Psi}\left(\boldsymbol{U}^{(0)}\right)=\boldsymbol{S}_{bc}\partial_t\Tilde{\boldsymbol{u}}_D$ and $\boldsymbol{\Psi}\left(\boldsymbol{U}^{(1)}\right)=\boldsymbol{0}$ on $\partial\Tilde{\Omega}\times(0,\Tilde{t}_f)$ (see Eq. \eqref{eq:bc_vector}).
\end{prop}

\begin{proof}
    In order to verify the conditions of a second-order consistent Dirichlet boundary formulation, we reuse the asymptotic expansion results from the domain interior, i.\,e. the results shown in Eqs. \eqref{eq:asymp_ic0} and \eqref{eq:asymp_ic1}. Note that $\boldsymbol{f}_{ij}^{(0)*}=\boldsymbol{f}_{ij}^{(0)}$ and $\boldsymbol{f}_{ij}^{(1)*}$ is obtained by inverting the sign of the latter two terms of $\boldsymbol{f}_{ij}^{(1)}$ in Eq. \eqref{eq:asymp_ic1}. For simplicity, let us first analyze the boundary formulation with only bounce-back or anti bounce-back, i.\,e. $\boldsymbol{D}=\boldsymbol{I}_n$ or $\boldsymbol{D}=-\boldsymbol{I}_n$, respectively, with $\boldsymbol{I}_n$ as the identity matrix in $n$ dimensions. If we introduce the expansion coefficients for the pre- and post-collision populations into the ansatz and Taylor expand around $(\boldsymbol{x}_b,t+\Delta t/2)$, we obtain for the two leading orders
    \begin{alignat}{4}
        \label{eq:abb0}
        &\text{Anti bounce-back}\qquad&\Delta t^0:\quad&\frac{1}{2}\boldsymbol{U}^{(0)} - \boldsymbol{S}_{ij} = \boldsymbol{0} \\
        \label{eq:abb1}
        &&\Delta t^1:\quad&\frac{1}{2}\boldsymbol{U}^{(1)} + \frac{1}{8}\left(q_{ij}-\frac{1}{2}\right)\left(i\partial_x \boldsymbol{U}^{(0)}+ j\partial_y \boldsymbol{U}^{(0)}\right) = \boldsymbol{0} \\
        \label{eq:bb0}
        &\text{Bounce-back}\qquad&\Delta t^0:\quad&i \Tilde{\Phi}_x\left(\boldsymbol{U}^{(0)}\right)+j \Tilde{\Phi}_y\left(\boldsymbol{U}^{(0)}\right) - \boldsymbol{S}_{ij} = \boldsymbol{0} \\
        \label{eq:bb1}
        &&\Delta t^1:\quad&i \Tilde{\Phi}_x\left(\boldsymbol{U}^{(1)}\right)+j \Tilde{\Phi}_y\left(\boldsymbol{U}^{(1)}\right) + \frac{1}{4}\left(q_{ij}-\frac{1}{2}\right)\left(i^2 \Tilde{\Phi}_x\left(\partial_x\boldsymbol{U}^{(0)}\right)+j^2 \Tilde{\Phi}_y\left(\partial_y\boldsymbol{U}^{(0)}\right)\right) = \boldsymbol{0},
    \end{alignat}
    where all functions are evaluated at $(\boldsymbol{x}_b,t+\Delta t/2)$. Combining the leading-order expressions for the pure anti bounce-back and bounce-back case (Eqs. \eqref{eq:abb0} and \eqref{eq:bb0}) to the mixed handling with $\boldsymbol{D}$ as defined in Eq. \eqref{eq:bc_d}, we obtain
    \begin{equation}
        \frac{1}{2}\boldsymbol{I}_u\boldsymbol{U}^{(0)} + i\boldsymbol{I}_{\Phi}\Tilde{\Phi}_x\left(\boldsymbol{U}^{(0)}\right) + j\boldsymbol{I}_{\Phi}\Tilde{\Phi}_y\left(\boldsymbol{U}^{(0)}\right) = \left[\begin{array}{rr}
    1/2 & 0 \\ 0 & 1/2 \\ i\Tilde{c}_K & j \Tilde{c}_K \\ i\Tilde{c}_{\mu} & -j\Tilde{c}_{\mu} \\ j\Tilde{c}_{\mu} & i\Tilde{c}_{\mu}
    \end{array}\right]\partial_t \Tilde{\boldsymbol{u}}_D,
    \end{equation}
    where we introduced the definition of $\boldsymbol{S}_{ij}$ according to Eq. \eqref{eq:bc_s}. For rectangular domains and the D2Q4$^5$ velocity set, it is straightforward to recognize that $i=-n_x$ and $j=-n_y$ on the boundary, so that the leading-order expansion satisfies the dimensionless form of Eq. \eqref{eq:bc_vector}. Continuing to the next order and combining the expansion results of Eqs. \eqref{eq:abb1} and \eqref{eq:bb1} yields
    \begin{equation}
        \frac{1}{2}\boldsymbol{I}_u\boldsymbol{U}^{(1)} + i\boldsymbol{I}_{\Phi}\Tilde{\Phi}_x\left(\boldsymbol{U}^{(1)}\right) + j\boldsymbol{I}_{\Phi}\Tilde{\Phi}_y\left(\boldsymbol{U}^{(1)}\right) = \boldsymbol{0},
    \end{equation}
    where we further applied Eq. \eqref{eq:bc_q}, i.\,e. $q_{ij}=\frac{1}{2}$ for all $(i,j)\in\mathcal{D}_{\boldsymbol{x}}$. Accordingly, we obtain a zero boundary condition for the next higher expansion coefficient and therefore the Dirichlet boundary formulation is second-order consistent.
\end{proof}

\section{Stability of the proposed method} \label{sec:stability}

Clearly, the removal of numerical dissipation makes the stability of the formulation a delicate issue. In this section, we analyze and demonstrate the stability of the proposed method. To this end, we perform a rigorous stability analysis based on the notion of pre-stability structures proposed in \cite{Banda2006,Junk2009,Rheinlander2010} and recently generalized towards vectorial LBM in \cite{Zhao2020,Zhao2020s,Zhao2024}. Additionally, we extend the stability result towards the case of Dirichlet problems by a modified proof for the streaming stage involving the boundary formulation, which is based on \cite{Junk2009,Zhao2020bc,Zhao2024}.

\subsection{Preliminaries}
\label{prelim}
Let $\boldsymbol{f}:=\left[\begin{matrix}\boldsymbol{f}^T_{10} & \boldsymbol{f}^T_{01} & \boldsymbol{f}^T_{\Bar{1}0} & \boldsymbol{f}^T_{0\Bar{1}}\end{matrix}\right]^T$ be the column vector of all vector-valued populations $\boldsymbol{f}:\mathbb{G}\times\mathbb{T}\rightarrow\mathbb{R}^{qn}$, where $n$ is the number of scalar components and $q$ the number of populations, i.\,e. $q:=|\mathcal{V}|$. We use Miller index notation to denote $-1$ as $\Bar{1}$. Using this population vector notation, we can rewrite the collision step  in Eq. \eqref{eq:collision} without forcing as follows
    \begin{equation}
    \label{eq:collision_vector}
        \boldsymbol{f}^* = \left(\boldsymbol{I}_{qn} + \omega(\boldsymbol{G}-\boldsymbol{I}_{qn})\right)\boldsymbol{f} =: (\boldsymbol{I}_{qn}+\boldsymbol{J})\boldsymbol{f},
    \end{equation}
    where $\boldsymbol{I}_{qn}\in\mathbb{R}^{qn \times qn}$ is the identity matrix and $\boldsymbol{G}\in\mathbb{R}^{qn \times qn}$ transforms the population vector into its equilibrium population vector according to Eq. \eqref{eq:feq}, i.\,e. $\boldsymbol{f}^{eq}:=\left[\begin{matrix}\boldsymbol{f}_{10}^{eq,T} & \boldsymbol{f}_{01}^{eq,T} & \boldsymbol{f}_{\Bar{1}0}^{eq,T} & \boldsymbol{f}_{0\Bar{1}}^{eq,T}\end{matrix}\right]^T = \boldsymbol{G}\boldsymbol{f}$. In our specific case, we have $q=4$ and $n=5$, according to D2Q4$^5$. Based on the definitions for $\Tilde{\boldsymbol{U}}^{num}$ in Eq. \eqref{eq:cond_eq} and the equilibrium populations in Eq. \eqref{eq:feq}, we can write $\boldsymbol{G}$ as follows:
    \begin{equation}
    \label{eq:stability_linear}
    \begin{split}
        &\boldsymbol{G} = \left[\begin{matrix}
            \boldsymbol{g}_{10} \\ \boldsymbol{g}_{01} \\ \boldsymbol{g}_{\Bar{1}0} \\ \boldsymbol{g}_{0\Bar{1}}
        \end{matrix}\right] \left[\begin{matrix}
            \boldsymbol{I}_n & \boldsymbol{I}_n & \boldsymbol{I}_n & \boldsymbol{I}_n
        \end{matrix}\right], \qquad \boldsymbol{g}_{ij} := \frac{1}{4}\boldsymbol{I}_n + \frac{1}{2c}(i\boldsymbol{A}_x + j\boldsymbol{A}_y) \qquad \forall (i,j)\in\mathcal{V} \\
        &\boldsymbol{A}_x := 
        \left[\begin{matrix}
            0 & 0 & \Tilde{c}_K & \Tilde{c}_{\mu} & 0 \\ 0 & 0 & 0 & 0 & \Tilde{c}_{\mu} \\ \Tilde{c}_K & 0 & 0 & 0 & 0 \\ \Tilde{c}_{\mu} & 0 & 0 & 0 & 0 \\ 0 & \Tilde{c}_{\mu} & 0 & 0 & 0
        \end{matrix}\right] \qquad \boldsymbol{A}_y := 
        \left[\begin{matrix}
            0 & 0 & 0 & 0 & \Tilde{c}_{\mu} \\ 0 & 0 & \Tilde{c}_K & -\Tilde{c}_{\mu} & 0 \\ 0 & \Tilde{c}_K & 0 & 0 & 0 \\ 0 & -\Tilde{c}_{\mu} & 0 & 0 & 0 \\ \Tilde{c}_{\mu} & 0 & 0 & 0 & 0
        \end{matrix}\right].
    \end{split}
    \end{equation}
    By the specific choice of the primary solution vector $\boldsymbol{U}$ in Eq. \eqref{eq:system}, we obtain symmetric matrices $\boldsymbol{A}_x$ and $\boldsymbol{A}_y$, which in turn lead to symmetric $\boldsymbol{g}_{ij},\; \forall (i,j)\in\mathcal{V}$. 

\subsection{Stability of vectorial LBM for elastodynamics with periodic boundaries}

We start by introducing the definition of pre-stability structures, which follows \cite{Banda2006,Junk2009,Rheinlander2010} and its generalization for vectorial LBM in \cite{Zhao2020s,Zhao2024}. 

\begin{definition}[Pre-stability structure] \label{def:prestability}
    The square matrix $\boldsymbol{J}\in\mathbb{R}^{qn\times qn}$ for some positive natural numbers $q,n$ is said to have a pre-stability structure if there exists an invertible matrix $\boldsymbol{P}\in\mathbb{R}^{qn\times qn}$ such that
    \begin{enumerate}[(i)]
        \item $\boldsymbol{P}\boldsymbol{J} = -\diag{(\lambda_1,\dots,\lambda_{qn})}\boldsymbol{P}$,
        \item $\boldsymbol{P}^T\boldsymbol{P} = \diag{(\boldsymbol{p}_1,\dots,\boldsymbol{p}_q)}$,
    \end{enumerate}
    with $\lambda_i\in\mathbb{R}$ for all $i=1,\dots,qn$ and symmetric positive definite matrices $\boldsymbol{p}_j\in\mathbb{R}^{n\times n}$ for all $j=1,\dots,q$.
\end{definition}

Comparing with the definition in Eq. \eqref{eq:collision_vector}, we recognize that $\boldsymbol{J}$ is the matrix representation of the (linear) collision operation, where $q$ corresponds to the number of populations and $n$ to the number of scalar components of each population.

Further, let us introduce a weighted L2 grid norm, which involves the invertible weighting matrix $\boldsymbol{P}$ already introduced in Definition \ref{def:prestability} and the vector of populations $\boldsymbol{f}$. The stability of the collision and streaming stages of LBM will be established using this norm.

\begin{definition}[Weighted L2 grid norm] \label{def:grid_norm}
     Let $\boldsymbol{P}\in\mathbb{R}^{qn\times qn}$ be a matrix satisfying all conditions in Definition \ref{def:prestability}. Then we define the weighted L2 grid norm of the population vector $\boldsymbol{f}$ as
    \begin{equation}
        ||\boldsymbol{f}||^2_{\boldsymbol{P}}(t) := \sum_{\boldsymbol{x}\in\mathbb{G}} |\boldsymbol{P}\boldsymbol{f}(\boldsymbol{x},t)|^2,
    \end{equation}
    where $|\bullet|$ denotes the Euclidean norm. We can also express the weighted L2 grid norm as
    \begin{equation}
        ||\boldsymbol{f}||^2_{\boldsymbol{P}}(t) = \sum_{\boldsymbol{x}\in\mathbb{G}} |\boldsymbol{P}\boldsymbol{f}(\boldsymbol{x},t)|^2 = \sum_{\boldsymbol{x}\in\mathbb{G}} \boldsymbol{f}^T\boldsymbol{P}^T\boldsymbol{P}\boldsymbol{f} = \sum_{\boldsymbol{x}\in\mathbb{G}} \sum_{(i,j)\in\mathcal{V}} \boldsymbol{f}_{ij}^T \boldsymbol{k}_{ij} \boldsymbol{f}_{ij} =: \sum_{\boldsymbol{x}\in\mathbb{G}} \sum_{(i,j)\in\mathcal{V}} |\boldsymbol{f}_{ij}(\boldsymbol{x},t)|^2_{\boldsymbol{k}_{ij}},
    \end{equation}
    introducing $q$ symmetric positive definite matrices $\boldsymbol{k}_{ij}\in\mathbb{R}^{n\times n}$ for all $(i,j)\in\mathcal{V}$, which induce well-defined norms, because of (ii) in Definition \ref{def:prestability}.
\end{definition}

Note that, instead of using the matrices $\boldsymbol{p}_j$ in Definition \ref{def:prestability}, we now refer to the symmetric positive definite diagonal blocks as $\boldsymbol{k}_{ij}$, which allows us to write the summation over the velocity index set $\mathcal{V}$.

In order to construct a weighting matrix $\boldsymbol{P}$ that satisfies the conditions in Definition \ref{def:prestability}, we further introduce the following lemma, that provides a more easily applicable condition for the construction of such a matrix. The statement and the proof are adopted from \cite{Zhao2020s,Zhao2024}, however, we attempt to add more detailed explanations.

\begin{lemma} \label{lem:1}
    Let $\boldsymbol{K}\in\mathbb{R}^{qn\times qn}$ be a block-diagonal matrix consisting of a number $q$ of symmetric positive definite blocks $\boldsymbol{k}_{ij}\in\mathbb{R}^{n\times n}$ for all $(i,j)\in\mathcal{V}$. If $\boldsymbol{K}$ is a symmetrizer of $\boldsymbol{J}\in\mathbb{R}^{qn\times qn}$, i.\,e. $\boldsymbol{K}\boldsymbol{J} = \boldsymbol{J}^T\boldsymbol{K}$ (i.\,e. $\boldsymbol{K}\boldsymbol{J}$ is symmetric), then $\boldsymbol{J}$ admits a pre-stability structure in the sense of Definition \ref{def:prestability}.
\end{lemma}

\begin{proof}
    Note that the square root of a symmetric positive definite matrix is symmetric and invertible. Therefore, by the assumptions $\boldsymbol{K}^{-1/2}\boldsymbol{K}\boldsymbol{J}\boldsymbol{K}^{-1/2}$ is symmetric, and therefore admits a spectral decomposition with a real-valued diagonal matrix $\boldsymbol{\Lambda}$ and an orthogonal matrix $\boldsymbol{U}$:
    \begin{equation}
    \label{eq:proof_lemma1}
        \boldsymbol{K}^{-1/2}\boldsymbol{K}\boldsymbol{J}\boldsymbol{K}^{-1/2} = \boldsymbol{U}^T\boldsymbol{\Lambda}\boldsymbol{U} \qquad \Leftrightarrow \qquad \boldsymbol{K}\boldsymbol{J} = \boldsymbol{K}^{1/2}\boldsymbol{U}^T\boldsymbol{\Lambda}\boldsymbol{U}\boldsymbol{K}^{1/2}.
    \end{equation}
    If we set $\boldsymbol{P} = \boldsymbol{U}\boldsymbol{K}^{1/2}$ we satisfy (ii) as follows: $\boldsymbol{P}^T\boldsymbol{P} = \boldsymbol{K}^{1/2}\boldsymbol{U}^T\boldsymbol{U}\boldsymbol{K}^{1/2} = \boldsymbol{K}$, where $\boldsymbol{K}$ consists of the symmetric positive definite block matrices as required. Further note that $\boldsymbol{P} = \boldsymbol{U}\boldsymbol{K}^{1/2}$ is by construction invertible. For (i), we start with the second expression in Eq. \eqref{eq:proof_lemma1}, introduce the definition of $\boldsymbol{P}$ and use the previously established relation (ii):
    \begin{equation}
        \boldsymbol{K}\boldsymbol{J} = \boldsymbol{K}^{1/2}\boldsymbol{U}^T\boldsymbol{\Lambda}\boldsymbol{U}\boldsymbol{K}^{1/2} \quad \Leftrightarrow \quad \boldsymbol{P}^T\boldsymbol{P}\boldsymbol{J}=\boldsymbol{P}^T\boldsymbol{\Lambda}\boldsymbol{P} \quad \Leftrightarrow \quad \boldsymbol{P}\boldsymbol{J}=\boldsymbol{\Lambda}\boldsymbol{P}.
    \end{equation}
    Finally, we can identify the real-valued diagonal matrix as follows: $\boldsymbol{\Lambda} = -\diag{(\lambda_1,\dots,\lambda_{qn})}$.
\end{proof}

With all definitions in place, we establish the following Proposition \ref{prop:stab_periodic} on periodic domains without body force $\Tilde{\boldsymbol{B}}$.

\begin{prop}[Stability with periodic boundary conditions] \label{prop:stab_periodic}
    In the absence of body forces, the lattice Boltzmann scheme on a periodic domain, which further satisfies all conditions of Proposition \ref{prop:consistency_interior}, is stable in the weighted L2 grid norm of Definition \ref{def:grid_norm}, i.\,e.
    \begin{equation}
        ||\boldsymbol{f}||_{\boldsymbol{P}}(t) \leq ||\boldsymbol{f}||_{\boldsymbol{P}}(0) \qquad \forall t\in\mathbb{T}.
    \end{equation}
\end{prop}

The proof is closely related to the works of \cite{Zhao2020s,Zhao2024}, but now applied to the case of linear elastodynamics. Note that our target problem is linear and therefore the proof does not involve any linearization and is therefore free of approximations or further assumptions.

\begin{proof}
    The evolution of the population vector $\boldsymbol{f}$ up to some arbitrary time $t\in\mathbb{T}$ corresponds to the repeated application of the collision and streaming stages defined in Eqs. \eqref{eq:collision} and \eqref{eq:streaming}, respectively. Therefore, it is sufficient to show that $||\boldsymbol{f}||_{\boldsymbol{P}}(t+\Delta t) \leq ||\boldsymbol{f}||_{\boldsymbol{P}}(t)$ for each arbitrary $t\in\mathbb{T}$. We shall demonstrate this in two stages: first the collision and then the streaming.

\emph{Collision. } The symmetric nature of the $\boldsymbol{g}_{ij}$ introduced in Section \ref{prelim} allows us to construct a symmetrizer $\boldsymbol{K}\in\mathbb{R}^{qn \times qn}$ as follows:
    \begin{equation}
    \label{eq:def_symmetrizer}
        \boldsymbol{K} := \diag{(\boldsymbol{k}_{10}, \boldsymbol{k}_{01}, \boldsymbol{k}_{\Bar{1}0}, \boldsymbol{k}_{0\Bar{1}})} = \diag{(\boldsymbol{g}^{-1}_{10}, \boldsymbol{g}^{-1}_{01}, \boldsymbol{g}^{-1}_{\Bar{1}0}, \boldsymbol{g}^{-1}_{0\Bar{1}})}.
    \end{equation}
    Note that in order to satisfy the conditions of Definition \ref{def:prestability} and Lemma \ref{lem:1}, i.\,e. symmetric positive definite $\boldsymbol{k}_{ij}$ for all $(i,j)\in\mathcal{V}$, we need
    \begin{equation}
    \label{eq:stab_cond}
        2\sqrt{\Tilde{c}_K^2+\Tilde{c}_{\mu}^2} < c,
    \end{equation}
    assuming positive speeds $\Tilde{c}_K$ and $\Tilde{c}_{\mu}$. Note that this corresponds exactly to the sub-characteristic condition or CFL condition that the fastest wave speed times the number of dimensions is slower than the numerical propagation speed $c=\Delta x/\Delta t$ \cite{Courant1928,Dubois2013,Graille2014}.

    It is straightforward to see that by construction $\boldsymbol{K}\boldsymbol{G}$ is symmetric and therefore also $\boldsymbol{K}\boldsymbol{J} = \omega(\boldsymbol{K}\boldsymbol{G}-\boldsymbol{K})$ (see Eq. \eqref{eq:collision_vector}). As a consequence of Lemma \ref{lem:1}, $\boldsymbol{J}$ admits a pre-stability structure according to Definition \ref{def:prestability}, which in a next step enables us to establish the contraction property of the collision stage.

    First, note that $\boldsymbol{G}$ is an orthogonal projector, i.\,e. $\boldsymbol{G}^2=\boldsymbol{G}$, with spectrum $\{0,1\}$, i.\,e. $\spec{(\boldsymbol{G})}=\{0,1\}$. To show this we consider the eigenvalue relation $\boldsymbol{G}\boldsymbol{w}_k=\phi_k \boldsymbol{w}_k$ involving an eigenvector $\boldsymbol{w}_k$ and corresponding eigenvalue $\phi_k$ for each $k=1,\dots,qn$. Therefore, we obtain
    \begin{equation}
        \phi_k \boldsymbol{w}_k = \boldsymbol{G}\boldsymbol{w}_k = \boldsymbol{G}^2\boldsymbol{w}_k = \boldsymbol{G}(\phi_k \boldsymbol{w}_k) = \phi_k^2 \boldsymbol{w}_k \qquad \forall k=1,\dots,qn.
    \end{equation}
    Comparing the leftmost and rightmost expressions, we conclude that $\phi_k=0$ or $\phi_k=1$ for all $k=1,\dots,qn$.
    
    Using the definition in Eq. \eqref{eq:collision_vector}, it follows that $\spec{(\boldsymbol{J})} = \spec{(\omega(\boldsymbol{G}-\boldsymbol{I}_{qn}))} = \{-\omega,0\}$. Because $\boldsymbol{J}$ admits a pre-stability structure according to Definition \ref{def:prestability}, we can use property (i) to conclude that there exists a similarity transformation between $\boldsymbol{J}$ and $-\diag{(\lambda_1,\dots,\lambda_{qn})}$ and therefore $\lambda_k \in\{0,\omega\}$ for all $k=1,\dots,qn$. Combining all the properties established so far, we can now show the contraction property of the collision stage. To this end, we take the weighted L2 grid norm (see Definition \ref{def:grid_norm}) of Eq. \eqref{eq:collision_vector} at some arbitrary time $t\in\mathbb{T}$:
    \begin{equation}
    \label{eq:bound_collision}
    \begin{split}
        ||\boldsymbol{f}^*||^2_{\boldsymbol{P}}(t) &= ||(\boldsymbol{I}_{qn}+\boldsymbol{J})\boldsymbol{f}||^2_{\boldsymbol{P}}(t) = \sum_{\boldsymbol{x}\in\mathbb{G}} |\boldsymbol{P}(\boldsymbol{I}_{qn}+\boldsymbol{J})\boldsymbol{f}|^2(t) \\
        &= \sum_{\boldsymbol{x}\in\mathbb{G}} |(\boldsymbol{I}_{qn}-\diag{(\lambda_1,\dots,\lambda_{qn})})\boldsymbol{P}\boldsymbol{f}|^2(t) \\
        &\leq \sum_{\boldsymbol{x}\in\mathbb{G}} |(\boldsymbol{I}_{qn}-\diag{(\lambda_1,\dots,\lambda_{qn})})|^2|\boldsymbol{P}\boldsymbol{f}|^2(t) \\
        &\leq \sum_{\boldsymbol{x}\in\mathbb{G}} \max_{k=1,\dots,qn}{|1-\lambda_k|^2|}\boldsymbol{P}\boldsymbol{f}|^2(t) \leq \sum_{\boldsymbol{x}\in\mathbb{G}} |\boldsymbol{P}\boldsymbol{f}|^2(t) = ||\boldsymbol{f}||^2_{\boldsymbol{P}}
    \end{split}
    \end{equation}
    Note that the existence of an invertible matrix $\boldsymbol{P}$ is guaranteed by Lemma \ref{lem:1} under condition Eq. \eqref{eq:stab_cond}. The last inequality requires $\lambda_k\in[0,2]$ for all $k=1,\dots,qn$. With the choice $\omega=2$, the collision step turns into an isometric operation under the weighted L2 grid norm.

    \emph{Streaming. } Turning to the streaming stage, we first observe that the streaming operation in Eq. \eqref{eq:streaming} including periodic boundary conditions (see Section \ref{sec:lbm}) is a bijection  from the lattice onto itself, i.\,e. $\chi: \mathbb{G}\times\mathcal{V}\rightarrow\mathbb{G}\times\mathcal{V}$. In other words, this means that for each population after streaming with lattice node position $\boldsymbol{y}\in\mathbb{G}$ and indices $(k,l)\in\mathcal{V}$, we can identify a unique position $\boldsymbol{x}\in\mathbb{G}$ and index $(i,j)\in\mathcal{V}$ where it came from. In the case of "normal" streaming in the domain interior, this position is one of the neighbor nodes, whereas at the periodic boundaries the incoming population on one side is the outgoing population at the opposing side. In both cases the velocity index remains unchanged. As a result, we obtain
    \begin{equation}
    \label{eq:bound_streaming}
        ||\boldsymbol{f}||^2_{\boldsymbol{P}}(t+\Delta t) = \sum_{\boldsymbol{y}\in\mathbb{G}} |\boldsymbol{P}\boldsymbol{f}(\boldsymbol{y},t+\Delta t)|^2 = \sum_{\boldsymbol{y}\in\mathbb{G}} \sum_{(i,j)\in\mathcal{V}} |\boldsymbol{f}_{ij}(\boldsymbol{y},t+\Delta t)|^2_{\boldsymbol{k}_{ij}} = \sum_{\boldsymbol{x}\in\mathbb{G}} \sum_{(i,j)\in\mathcal{V}} |\boldsymbol{f}^*_{ij}(\boldsymbol{x},t)|^2_{\boldsymbol{k}_{ij}} = ||\boldsymbol{f}^*||^2_{\boldsymbol{P}}(t),
    \end{equation}
    which implies isometry of the streaming stage under the weighted L2 grid norm. Note how the property (ii) of Definition \ref{def:prestability} is applied, which relies on the positive definite block-diagonal structure of $\boldsymbol{K}$. Furthermore, the existence of the bijection $\chi$ allows us to replace the finite sum over the lattice $\mathbb{G}$ of the post-streaming state with an equivalent sum at the pre-streaming state.

    Combining the results of Eqs. \eqref{eq:bound_collision} and \eqref{eq:bound_streaming} for a single time step yields
    \begin{equation}
        ||\boldsymbol{f}||_{\boldsymbol{P}}(t+\Delta t) = ||\boldsymbol{f}^*||_{\boldsymbol{P}}(t) = ||\boldsymbol{f}||_{\boldsymbol{P}}(t)
    \end{equation}
    for arbitrary $t\in\mathbb{T}$, with $\omega=2$ and the condition in Eq. \eqref{eq:stab_cond}.
\end{proof}

Note that a similar proof of Proposition \ref{prop:stab_periodic} is possible with nonzero forcing $\Tilde{\boldsymbol{B}}$ if $\Tilde{\boldsymbol{B}}$ is bounded for all $(\boldsymbol{x},t)\in\Tilde{\Omega}\times(0,\Tilde{t}_f)$. The statement then involves a constant, which depends on this bound:
\begin{equation}
\label{eq:stab_forcing}
        ||\boldsymbol{f}||_{\boldsymbol{P}}(t) \leq C||\boldsymbol{f}||_{\boldsymbol{P}}(0) \qquad \forall t\in\mathbb{T} \qquad \text{for some constant }C\in\mathbb{R} \text{ independent of }\mathbb{G}\times\mathbb{T}.
\end{equation}

\subsection{Stability of vectorial LBM for elastodynamics with Dirichlet boundary conditions} \label{sec:bcs}
Lastly, we investigate the stability of the LBM formulation in the presence of Dirichlet boundary conditions. For simplicity, we assume homogeneous boundary conditions. However, a similar statement is possible with nonzero Dirichlet boundary conditions $\Tilde{\boldsymbol{u}}_D$ (and nonzero forcing $\Tilde{\boldsymbol{B}}$) if we introduce some constant $C\in\mathbb{R}$ that depends on upper bounds for both sources of data (see Eq. \eqref{eq:stab_forcing}).

\begin{prop}[Stability with Dirichlet boundary conditions] \label{prop:stability_dbc}
    In the absence of body forces and for homogeneous Dirichlet boundary conditions $\Tilde{\boldsymbol{u}}_D$ , the lattice Boltzmann scheme, which satisfies all conditions of Propositions \ref{prop:consistency_interior} and \ref{prop:consistency_dbc}, is stable in the weighted L2 grid norm of Definition \ref{def:grid_norm}, i.\,e.
    \begin{equation}
        ||\boldsymbol{f}||_{\boldsymbol{P}}(t) \leq ||\boldsymbol{f}||_{\boldsymbol{P}}(0) \qquad \forall t\in\mathbb{T}.
    \end{equation}
\end{prop}

\begin{proof}
    The proof is the same as for Proposition \ref{prop:stab_periodic} except for the streaming portion of the algorithm. With the presence of the boundary formulation, we separate the handling of normal streaming (Eq. \eqref{eq:streaming}) and the boundary formulation (Eq. \eqref{eq:bc_ansatz}). Again, for all populations undergoing normal streaming this operation is described by a bijection $\chi: \mathbb{G}\times(\mathcal{V}\setminus(-\mathcal{D}_{\boldsymbol{x}})) \rightarrow\mathbb{G}\times(\mathcal{V}\setminus\mathcal{D}_{\boldsymbol{y}})$ for each lattice node $\boldsymbol{x},\boldsymbol{y}\in\mathbb{G}$. The domain of $\chi$ contains all lattice nodes and velocity indices apart from the ones that would be streamed outside the domain, whereas the range encompasses all nodes and all indices except the ones that would require streaming from outside the domain. Note that $\mathcal{D}_{\boldsymbol{x}}=-\mathcal{D}_{\boldsymbol{x}}=\emptyset$ for each $\boldsymbol{x}\in\mathbb{G}\setminus\mathbb{B}$, which allows us to conclude the following:
    \begin{equation}
    \label{eq:stab_bc_streaming}
        \sum_{\boldsymbol{y}\in\mathbb{G}} \sum_{(i,j)\in\mathcal{V}\setminus\mathcal{D}_{\boldsymbol{y}}} |\boldsymbol{f}_{ij}(\boldsymbol{y},t+\Delta t)|^2_{\boldsymbol{k}_{ij}} = \sum_{\boldsymbol{x}\in\mathbb{G}} \sum_{(i,j)\in\mathcal{V}\setminus(-\mathcal{D}_{\boldsymbol{x}})} |\boldsymbol{f}^*_{ij}(\boldsymbol{x},t)|^2_{\boldsymbol{k}_{ij}} = \sum_{\boldsymbol{x}\in\mathbb{G}\setminus\mathbb{B}} \sum_{(i,j)\in\mathcal{V}} |\boldsymbol{f}^*_{ij}(\boldsymbol{x},t)|^2_{\boldsymbol{k}_{ij}} + \sum_{\boldsymbol{x}\in\mathbb{B}} \sum_{(i,j)\in\mathcal{V}\setminus(-\mathcal{D}_{\boldsymbol{x}})} |\boldsymbol{f}^*_{ij}(\boldsymbol{x},t)|^2_{\boldsymbol{k}_{ij}}.
    \end{equation}

    For the populations $(i,j)\in\mathcal{D}_{\boldsymbol{x}}$ at all boundary nodes $\boldsymbol{x}\in\mathbb{B}$ we adopt the boundary formulation from Eq. \eqref{eq:bc_ansatz} with $\boldsymbol{D} := \diag(-1,-1,1,1,1)$ and $\boldsymbol{S}_{ij}=\boldsymbol{0}$ for the homogeneous boundary condition. Note that the boundary formulation is a bijection of the form $\psi: \mathbb{B}\times(-\mathcal{D}_{\boldsymbol{x}})\rightarrow\mathbb{B}\times\mathcal{D}_{\boldsymbol{x}}$ for each $\boldsymbol{x}\in\mathbb{B}$. Together with this observation, the crucial property that allows us to conclude the stability of the boundary formulation is
    \begin{equation}
        \boldsymbol{k}_{\Bar{\imath}\Bar{\jmath}} = \boldsymbol{D}^T\boldsymbol{k}_{ij}\boldsymbol{D} \qquad \forall (i,j)\in\mathcal{V},
    \end{equation}
    which follows from the definitions in Eqs. \eqref{eq:stability_linear} and \eqref{eq:def_symmetrizer}. Therefore, we obtain for all missing incoming populations
    \begin{equation}
    \label{eq:stab_bc_bb}
    \begin{split}
        \sum_{\boldsymbol{x}\in\mathbb{B}} \sum_{(i,j)\in\mathcal{D}_{\boldsymbol{x}}} |\boldsymbol{f}_{ij}(\boldsymbol{x},t+\Delta t)|^2_{\boldsymbol{k}_{ij}} &= \sum_{\boldsymbol{x}\in\mathbb{B}} \sum_{(\Bar{\imath},\Bar{\jmath})\in(-\mathcal{D}_{\boldsymbol{x}})} |\boldsymbol{D}\boldsymbol{f}^*_{\Bar{\imath}\Bar{\jmath}}(\boldsymbol{x},t)|^2_{\boldsymbol{k}_{ij}} = \sum_{\boldsymbol{x}\in\mathbb{B}} \sum_{(\Bar{\imath},\Bar{\jmath})\in(-\mathcal{D}_{\boldsymbol{x}})} \boldsymbol{f}^{*T}_{\Bar{\imath}\Bar{\jmath}}(\boldsymbol{x},t)\boldsymbol{D}^T\boldsymbol{k}_{ij}\boldsymbol{D}\boldsymbol{f}^*_{\Bar{\imath}\Bar{\jmath}}(\boldsymbol{x},t) \\
        &=\sum_{\boldsymbol{x}\in\mathbb{B}} \sum_{(\Bar{\imath},\Bar{\jmath})\in(-\mathcal{D}_{\boldsymbol{x}})} \boldsymbol{f}^{*T}_{\Bar{\imath}\Bar{\jmath}}(\boldsymbol{x},t)\boldsymbol{k}_{\Bar{\imath}\Bar{\jmath}}\boldsymbol{f}^*_{\Bar{\imath}\Bar{\jmath}}(\boldsymbol{x},t) = \sum_{\boldsymbol{x}\in\mathbb{B}} \sum_{(\Bar{\imath},\Bar{\jmath})\in(-\mathcal{D}_{\boldsymbol{x}})} |\boldsymbol{f}^*_{\Bar{\imath}\Bar{\jmath}}(\boldsymbol{x},t)|^2_{\boldsymbol{k}_{\Bar{\imath}\Bar{\jmath}}}.
    \end{split}
    \end{equation}
    Combining the results of Eqs. \eqref{eq:stab_bc_streaming} and \eqref{eq:stab_bc_bb} yields 
    \begin{equation}
        \begin{split}
            ||\boldsymbol{f}||^2_{\boldsymbol{P}}(t+\Delta t) &= \sum_{\boldsymbol{y}\in\mathbb{G}} \sum_{(i,j)\in\mathcal{V}\setminus\mathcal{D}_{\boldsymbol{y}}} |\boldsymbol{f}_{ij}(\boldsymbol{y},t+\Delta t)|^2_{\boldsymbol{k}_{ij}} + \sum_{\boldsymbol{x}\in\mathbb{B}} \sum_{(i,j)\in\mathcal{D}_{\boldsymbol{x}}} |\boldsymbol{f}_{ij}(\boldsymbol{x},t+\Delta t)|^2_{\boldsymbol{k}_{ij}} \\
            &= \sum_{\boldsymbol{x}\in\mathbb{G}\setminus\mathbb{B}} \sum_{(i,j)\in\mathcal{V}} |\boldsymbol{f}^*_{ij}(\boldsymbol{x},t)|^2_{\boldsymbol{k}_{ij}} + \sum_{\boldsymbol{x}\in\mathbb{B}} \sum_{(i,j)\in\mathcal{V}\setminus(-\mathcal{D}_{\boldsymbol{x}})} |\boldsymbol{f}^*_{ij}(\boldsymbol{x},t)|^2_{\boldsymbol{k}_{ij}} + \sum_{\boldsymbol{x}\in\mathbb{B}} \sum_{(\Bar{\imath},\Bar{\jmath})\in(-\mathcal{D}_{\boldsymbol{x}})} |\boldsymbol{f}^*_{\Bar{\imath}\Bar{\jmath}}(\boldsymbol{x},t)|^2_{\boldsymbol{k}_{\Bar{\imath}\Bar{\jmath}}} = ||\boldsymbol{f}^*||^2_{\boldsymbol{P}}(t),
        \end{split}
    \end{equation}
    which shows that streaming together with the Dirichlet boundary formulation (with a homogeneous boundary condition) is an isometric operation in the weighted L2 grid norm. Together with the already established result for the collision stage (see proof of Proposition \ref{prop:stab_periodic}), we achieve the proposed stability property.
\end{proof}

\section{Numerical verification} \label{sec:verification}

For the numerical verification of the theoretical derivations, in this section we carry out a series of convergence studies using the method of manufactured solutions. We also perform long-term simulations in order to numerically verify the predicted stability properties.

\subsection{Convergence studies} \label{sec:num_convergence}

We employ the method of manufactured solutions to investigate the convergence order of the novel lattice Boltzmann formulation for linear elastodynamics on periodic domains as well as for Dirichlet problems. To this end, we assume a certain ansatz for the displacement solution $\hat{\boldsymbol{u}}: \Bar{\Omega}\times[0,\Tilde{t}_f)\rightarrow\mathbb{R}^2$, which we define on the interior and at the boundary of the domain, i.\,e. on $\Bar{\Omega}:=\Tilde{\Omega}\cup\partial\Tilde{\Omega}$, as well as for $t=0$. Next, we introduce $\hat{\boldsymbol{u}}$ into the governing equation which yields the corresponding body load and initial conditions, as well as boundary conditions (in the case of the Dirichlet problem) for a given set of material constants $K$ and $\mu$ and therefore speeds $c_K$ and $c_{\mu}$:
\begin{align}
    \boldsymbol{b} &= \left.\left(\partial_t^2\hat{\boldsymbol{u}} - \nabla \cdot \boldsymbol{\sigma}\left(\hat{\boldsymbol{u}}\right)\right)\right|_{\Omega\times(0,t_f)} \\
    \boldsymbol{u}_0 &= \left.\hat{\boldsymbol{u}}\right|_{\Omega\times\{0\}} \\
    \boldsymbol{v}_0 &= \left.\partial_t\hat{\boldsymbol{u}}\right|_{\Omega\times\{0\}} \\
    \boldsymbol{u}_D &= \left.\hat{\boldsymbol{u}}\right|_{\partial\Omega\times(0,t_f)}.
\end{align}
We denote the restriction of $\hat{\boldsymbol{u}}$ to a certain domain  by $\hat{\boldsymbol{u}}|_{(\bullet)}$.
Using the data, we carry out a series of simulations on consecutively finer discretizations according to acoustic scaling, i.\,e. $\Delta t \sim \Delta x$, and evaluate the numerical approximation error in the L2 and Linf norms defined below:
\begin{alignat}{3}
    \text{L2}\left(\boldsymbol{e}^{\phi}\right) &:= \left(\Delta x^2\Delta t\sum_{(\boldsymbol{x},t)\in\mathbb{G}\times\mathbb{T}}|\boldsymbol{e}^{\phi}(\boldsymbol{x},t)|^2\right)^{\tfrac{1}{2}} &&\\
    \text{Linf}\left(\boldsymbol{e}^{\phi}\right) &:= \max_{(\boldsymbol{x},t)\in\mathbb{G}\times\mathbb{T}} |\boldsymbol{e}^{\phi}(\boldsymbol{x},t)|_{\infty} &&\\
    \text{with } \boldsymbol{e}^{\phi} &:= \phi^{num} - \hat{\phi} &&\qquad \forall \phi\in\{\boldsymbol{u},{\boldsymbol{\sigma}}\}.
\end{alignat}
In the norm definitions, $|\bullet|$ once again refers to the Euclidean norm, whereas $|\bullet|_{\infty}$ is the maximum norm. As described in Section \ref{sec:lbm_postprocessing}, the displacement approximation $\boldsymbol{u}^{num}$ is obtained through an additional time integration step. For the Cauchy stress, we adopt a vectorized arrangement of the independent components for both the exact solution and the numerical approximation, i.\,e. $\Hat{\boldsymbol{\sigma}} := \left[\begin{matrix}\sigma_{xx}(\hat{\boldsymbol{u}})&\sigma_{yy}(\hat{\boldsymbol{u}})&\sigma_{xy}(\hat{\boldsymbol{u}})\end{matrix}\right]^T$, see also Eq. \eqref{eq:sigma_num}.
For better comparison, we also normalize the error norms by the L2 norm of the exact solution to obtain the relative L2 and Linf norm:
\begin{equation}
    \begin{alignedat}{2}
        \text{L2rel}\left(\boldsymbol{e}^{\phi}\right) &:= \text{L2}\left(\boldsymbol{e}^{\phi}\right)/\text{L2}\left(\Hat{\phi}\right) \\
        \text{Linfrel}\left(\boldsymbol{e}^{\phi}\right) &:= \text{Linf}\left(\boldsymbol{e}^{\phi}\right)/\text{L2}\left(\Hat{\phi}\right)
    \end{alignedat} \qquad \forall\phi\in\{\boldsymbol{u},{\boldsymbol{\sigma}}\}.
\end{equation}
From the second-order consistency and stability demonstrated in Sections \ref{sec:analysis} and \ref{sec:stability}, we expect second-order convergence of the error in both the L2 and Linf norm w.\,r.\,t. the target equation in the limit of vanishing time step size $\Delta t$ and grid spacing $\Delta x$.

\begin{figure}[htb]
    \centering
    \includegraphics{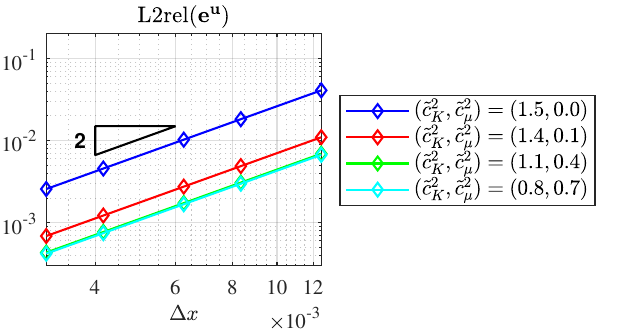} \hspace{-5.02cm}
    \includegraphics{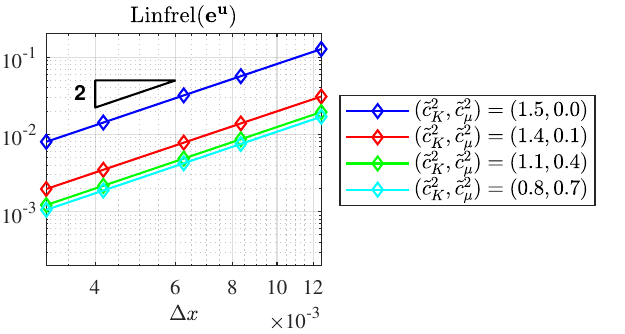}
    \caption{Convergence study with periodic boundary conditions. Comparison of displacement error for multiple material parameter combinations.}
    \label{fig:convergence_disp_per}
\end{figure}

\begin{figure}[htb]
    \centering
    \includegraphics{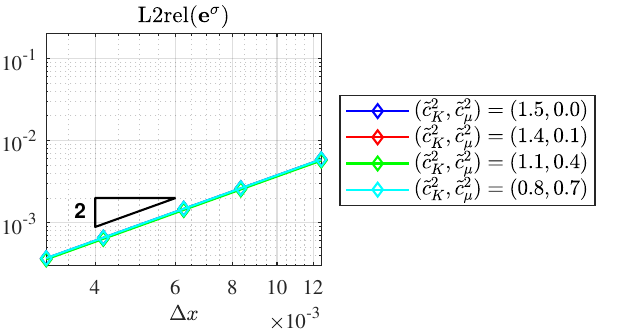} \hspace{-5.02cm}
    \includegraphics{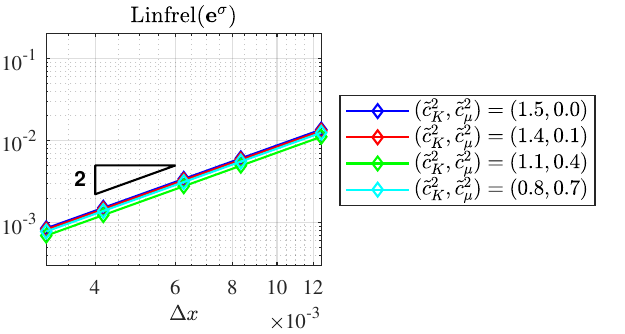}
    \caption{Convergence study with periodic boundary conditions. Comparison of Cauchy stress error for multiple material parameter combinations.}
    \label{fig:convergence_stress_per}
\end{figure}

First, we perform a convergence study on a square periodic domain $\Omega=(0,L)^2$ and $t_f=T$ and evaluate the method for multiple combinations of $\Tilde{c}_K$ and $\Tilde{c}_{\mu}$. Note that all material parameters are chosen in such a way as to maintain the same maximum wave speed according to Eq. \eqref{eq:stab_cond}. The $L$-periodic manufactured solution is given by
\begin{equation}
\label{eq:periodic_moving}
    \hat{\boldsymbol{u}} = \left[\begin{matrix}
        \sin{(4\pi(x-0.3t))}\cos{(2\pi(y-0.8t))}\sin{(4\pi(t-0.1))} \\
        \cos{(4\pi(x-0.7t))}\sin{(2\pi(y-0.1t))}\cos{(4\pi(t+0.4))}
    \end{matrix}\right] \qquad \forall (\boldsymbol{x},t)\in\bar{\Omega}\times[0,\Tilde{t}_f),
\end{equation}
with $\boldsymbol{x}=x\boldsymbol{e}_x+y\boldsymbol{e}_y$. The convergence is performed on lattices with grid spacing $\Delta x$ and time step sizes $\Delta t$ as follows: $(\Delta x,\Delta t)\in\{(1/80,1/200),(1/120,1/300),(1/160,1/400),(1/240,1/600),(1/320,1/800)\}$, thus satisfying the acoustic scaling assumption.

Figures \ref{fig:convergence_disp_per} and \ref{fig:convergence_stress_per} show the error in the displacement and Cauchy stress solution measured in both norms, respectively. As predicted by the analysis in Section \ref{sec:analysis}, we obtain second-order convergence for both solution fields in both norms. However, Figure \ref{fig:convergence_disp_per} shows that the error constant for the displacement field increases with an increase of the ratio of bulk and shear modulus. Conversely, the error in the Cauchy stress solution appears to be largely independent of the chosen material parameters and shows overall slightly lower values than the error in the displacement
solution.

For better comparability, we perform the exact same analysis, but only replace the periodicity by Dirichlet boundary conditions on the whole boundary. The results of the convergence studies are summarized in Figures \ref{fig:convergence_disp_dbc} and \ref{fig:convergence_stress_dbc}. Once again, we obtain second-order convergence of the error in the displacement solution in both norms, which is a strong confirmation of the analysis. Further, we can notice a slightly higher $\text{Linf}\left(\boldsymbol{e}^{\boldsymbol{u}}\right)$ when compared to the periodic case, which can be attributed to additional error contributions due to the boundary formulation. The error of the Cauchy stress solution given in Figure \ref{fig:convergence_stress_dbc} now shows a different behavior. Whereas the L2-norm still shows approximately second-order convergence, we observe a reduced (first-order) convergence in the Linf-norm. This lets us conclude that the Dirichlet boundary formulation introduces a first-order error in the Cauchy stress, which remains localized at the boundary. Note that for the material parameters $(\Tilde{c}_K^2,\Tilde{c}_{\mu}^2)\in\{(1.5,0.0),(1.4,0.1)\}$ the convergence order degrades only  for relatively fine discretizations to first order -- probably because other second-order errors with a higher constant overshadow the first-order error for finite $\Delta x$ (or equivalently $\Delta t$). A similar behavior was already observed with Dirichlet boundary conditions in the linear elastostatics case and was extensively analyzed in \cite{Boolakee2023a}. Overall, also with Dirichlet boundary conditions, we retain the second-order convergence in both solution fields in the L2 norm, i.\,e. in an average sense. Note that, for similar problems, the finite element method with linear ansatz functions  achieves second-order convergence for the displacement solution but only first-order convergence for the Cauchy stress solution \cite{Hughes2000}.

\begin{figure}[htb]
    \centering
    \includegraphics{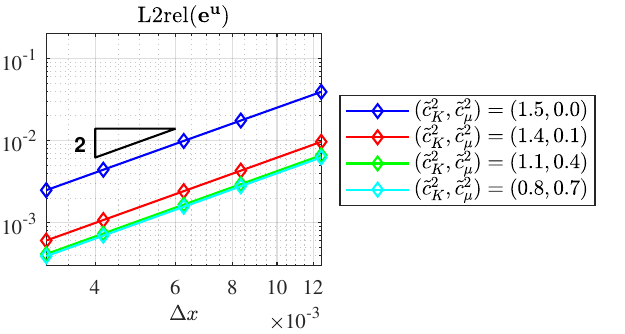} \hspace{-5.02cm}
    \includegraphics{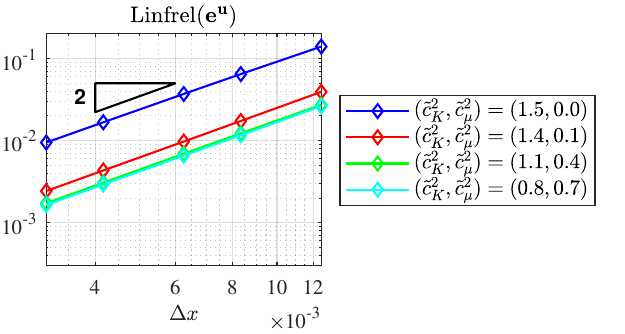}
    \caption{Convergence study with Dirichlet boundary conditions. Comparison of displacement error for multiple material parameter combinations.}
    \label{fig:convergence_disp_dbc}
\end{figure}

\begin{figure}[htb]
    \centering
    \includegraphics{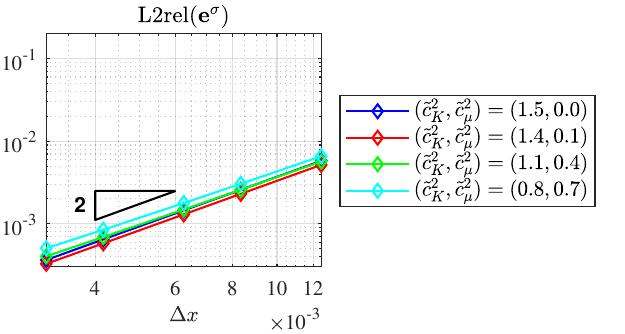} \hspace{-5.02cm}
    \includegraphics{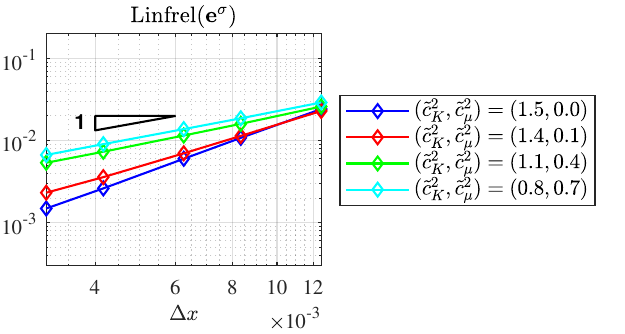}
    \caption{Convergence study with Dirichlet boundary conditions. Comparison of Cauchy stress error for multiple material parameter combinations.}
    \label{fig:convergence_stress_dbc}
\end{figure}

\subsection{Numerical stability tests}

In order to numerically verify the stability results of Propositions \ref{prop:stab_periodic} and \ref{prop:stability_dbc}, we track the evolution of the weighted L2 grid norm of the population vector during the numerical simulation. Because the case with the Dirichlet boundary condition builds on top of the stability result for periodic domains, we only demonstrate numerical examples with Dirichlet boundary conditions.

We again consider the domain $\Omega=(0,L)^2$ and take $t_f=10T$. We initialize with
\begin{align}
    \boldsymbol{u}_0(\boldsymbol{x}) = \left[\begin{matrix}
        \sin{(4\pi x)}\sin{(2\pi y)}\sin{(4\pi(t-0.1))} \\
        \sin{(4\pi x)}\sin{(2\pi y)}\sin{(4\pi(t+0.4))}
    \end{matrix}\right],\quad \boldsymbol{v}_0(\boldsymbol{x})=\partial_t\boldsymbol{u}_0(\boldsymbol{x}) \qquad \forall \boldsymbol{x}\in\Omega,
\end{align}
which is compatible with homogeneous Dirichlet boundary conditions. As material parameters we set $(\Tilde{c}_K^2,\Tilde{c}_{\mu}^2)=(1.1,0.4)$ and use a discretization with $(\Delta x, \Delta t)=(1/160,1/400)$. Therefore, we satisfy the stability criterion of Eq. \eqref{eq:stab_cond}: $2\sqrt{\Tilde{c}_K^2+\Tilde{c}_{\mu}^2}/c \approx 0.98 < 1$. Additionally, we set $\boldsymbol{b}=\boldsymbol{0}$ in order to satisfy all conditions under which we established the result in Proposition \ref{prop:stability_dbc}. For the given setup we numerically observe
\begin{equation}
    \max_{t\in\mathbb{T}}\frac{\left|\,||\boldsymbol{f}||_{\boldsymbol{P}}(t)-||\boldsymbol{f}||_{\boldsymbol{P}}(0)\right|}{||\boldsymbol{f}||_{\boldsymbol{P}}(0)} \approx 3.4\cdot 10^{-15},
\end{equation}
which confirms the analysis up to machine precision. Note that the seemingly trivial initial condition along with the homogeneous Dirichlet boundary condition produces a highly complex and unsteady solution field with a dense spectrum, for which we are still observing the predicted behavior with optimal accuracy.

Next, we again solve the Dirichlet problem with the manufactured solution reported in Eq. \eqref{eq:periodic_moving}. The domain is again $(0,L)^2$, this time with $t_f=T$, and $(\Tilde{c}_K^2,\Tilde{c}_{\mu}^2)=(1.4,0.1)$. As opposed to the previous example, we now prescribe a nonzero body load $\boldsymbol{b}$ and inhomogeneous Dirichlet boundary conditions $\boldsymbol{u}_D$ and therefore do not conserve the population vector $\boldsymbol{f}$ in the weighted L2 grid norm. However, $\boldsymbol{f}$ should still remain bounded, with the bound only depending on the data and not on the discretization parameters. Therefore, we compare the time evolution of $||\boldsymbol{f}(t)||_{\boldsymbol{P}}$ on two different lattices with $(\Delta x,\Delta t)\in\{(1/160,1/400),(1/320,1/800)\}$.

\begin{figure}[htb]
    \centering
    \includegraphics{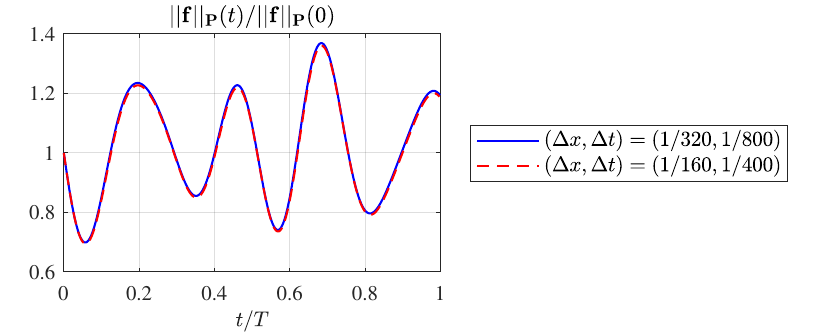}
    \caption{Comparison of $||\boldsymbol{f}||_{\boldsymbol{P}}$ with nonzero body load and Dirichlet boundary conditions on two different discretizations.}
    \label{fig:stability_fnorm}
\end{figure}

Figure \ref{fig:stability_fnorm} shows that the evolution is in close agreement for the two discretizations. However, a small and over time increasing offset towards lower values is observed for the coarser discretization. A straightforward explanation is based on the following reasoning. Assuming that the finer discretization is obtained by halving the lattice spacing and time step size of the coarser one (or in general by dividing by a positive integer), we see that
\begin{equation}
    (\mathbb{G}\times\mathbb{T})_{\text{coarse}} \subset (\mathbb{G}\times\mathbb{T})_{\text{fine}} \quad \Rightarrow \quad \sup_{(\boldsymbol{x},t)\in(\mathbb{G}\times\mathbb{T})_{\text{coarse}}}\Tilde{\boldsymbol{B}}(\boldsymbol{x},t) \leq \sup_{(\boldsymbol{x},t)\in(\mathbb{G}\times\mathbb{T})_{\text{fine}}}\Tilde{\boldsymbol{B}}(\boldsymbol{x},t) \leq \sup_{(\boldsymbol{x},t)\in\Tilde{\Omega}\times(0,\Tilde{t}_f)}\Tilde{\boldsymbol{B}}(\boldsymbol{x},t),
\end{equation}
and similarly for the Dirichlet boundary condition $\Tilde{\boldsymbol{u}}_D$. Note that the last inequality holds for any discretization $\mathbb{G}\times\mathbb{T}$ and therefore we can always bound the maximum possible increase of $||\boldsymbol{f}||_{\boldsymbol{P}}$ on any chosen discretization if the data itself is bounded.

Lastly, we perform an analysis using the manufactured solution in Eq. \eqref{eq:periodic_moving} with $(\Delta x, \Delta t)=(1/160,1/400)$ up until $t_f=2500T$, which results in $10^6$ time steps on $25600$ nodes. In order to assess the time evolution of the error, we define the relative L2 norm of the displacement error integrated only in space, i.\,e.
\begin{equation}
    \begin{split}
        \text{L2rel}^{\Omega}\left(\boldsymbol{e}^{\boldsymbol{u}}\right) &:= \text{L2}^{\Omega}\left(\boldsymbol{e}^{\boldsymbol{u}}\right)/\text{L2}^{\Omega}\left(\hat{\boldsymbol{u}}\right) \\
        \text{with } \text{L2}^{\Omega}\left(\boldsymbol{e}^{\boldsymbol{u}}\right) &:= \left(\Delta x^2\sum_{(\boldsymbol{x},t)\in\mathbb{G}}|\boldsymbol{e}^{u}(\boldsymbol{x},t)|^2\right)^{\tfrac{1}{2}}.
    \end{split}
\end{equation}

\begin{figure}[htb]
    \centering
    \includegraphics{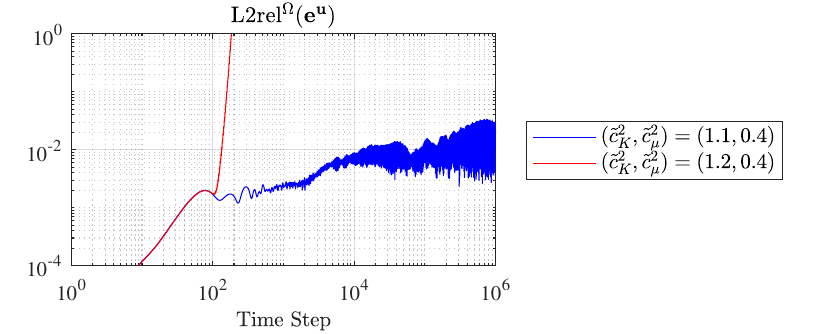}
    \caption{Long-term evolution of the numerical error for stable and unstable parametrizations.}
    \label{fig:stability_long}
\end{figure}

Accordingly, Figure \ref{fig:stability_long} shows the long-term evolution of this error norm for two different experiments with $(\Tilde{c}_K^2,\Tilde{c}_{\mu}^2)\in\{(1.1,0.4),(1.2,0.4)\}$. For the given discretization parameters this results in one parametrization which satisfies Eq. \eqref{eq:stab_cond} and another which does not:
\begin{equation}
    \frac{2\sqrt{\Tilde{c}_K^2+\Tilde{c}_{\mu}^2}}{c} = \frac{2\sqrt{1.1+0.4}}{2.5} \approx 0.98 < 1 \text{ (stable)}, \qquad\qquad \frac{2\sqrt{1.2+0.4}}{2.5} \approx 1.01 > 1 \text{ (unstable)}.
\end{equation}

As expected, for the parameters that satisfy Eq. \eqref{eq:stab_cond} stability is retained for up to $10^6$ time steps, whereas the other setup leads to divergence within a few hundred steps. Looking more closely at the error evolution for the stable simulation we observe a slow but gradual increase in the error, which can be attributed to the build-up of numerical approximation errors scaling as $\mathcal{O}\left(\Delta t^2\right)$. In Figure \ref{fig:stability_cut}, we compare the exact and numerical displacement solution along an arbitrary horizontal cut at $t=2500T$ (corresponding to $5000$ cycles of the solution in time). The good agreement after $10^6$ time steps further serves as a strong verification of the excellent non-dissipative behavior of the method (note that the forcing computed by the method of manufactured solutions is solely based on the exact solution and is therefore unable to possibly compensate any kind of dissipation occurring during the solution process).

\begin{figure}[htb]
    \centering
    \includegraphics{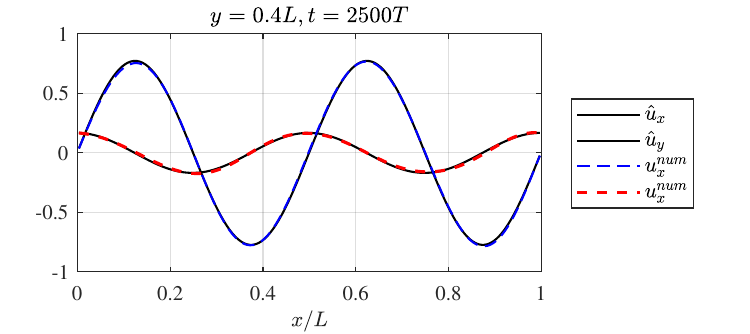}
    \caption{Comparison of exact and numerical displacement solution after $10^6$ time steps.}
    \label{fig:stability_cut}
\end{figure}

\section{Conclusions}

We propose a novel vectorial LBM formulation for linear elastodynamics on 2D rectangular domains with periodic or Dirichlet boundary conditions. 
For the derivation of the new LBM formulation, we first reformulate the target equation as a first-order hyperbolic system of equations. For this equivalent problem statement we construct a LBM scheme using vector-valued populations, where the evolution of each scalar component of the populations approximates one of the scalar equations of the system. Because the individual equations of the hyperbolic system are relatively simple in nature, it is sufficient to employ only four populations, which results in the D2Q4$^5$ velocity set for our problem. Using the asymptotic expansion technique, we establish second-order consistency of the novel formulation under the condition that second-order consistent initial and boundary formulations are used. Extending the asymptotic expansion results of the domain interior to the initial time and to domain boundaries enables us to systematically derive the corresponding initial and boundary formulations. Both with and without Dirichlet boundary conditions, we further derive a CFL-like stability condition using the notion of pre-stability structures extended to vectorial LBM. All results are verified by numerical experiments using the method of manufactured solutions and long-term stability tests.

With this initial contribution, we obtained a novel LBM formulation for linear elastodynamics that overcomes the shortcomings of previous works -- by providing second-order consistency and rigorous stability estimates. Furthermore, note that the definition of $\boldsymbol{\Phi}_x$ and $\boldsymbol{\Phi}_y$ in Eq. \eqref{eq:system} can be generalized to arbitrary nonlinear behavior as well and therefore a straightforward extension to nonlinear material behavior and/or large deformation settings appears feasible. In this contribution, linearity of the problem has only been exploited in the consistency analysis and construction of the stability proof. In consistency analysis, nonlinear behavior can be easily accommodated, whereas the stability analysis potentially requires a more general treatment for nonlinear problems.

The main limitation of the present formulation is that it can only be used for periodic or Dirichlet problems with grid-aligned boundaries. Therefore, the next natural step will be to generalize the boundary formulation in terms of both Neumann boundary conditions and arbitrary geometries. Already in the context of "conventional" LBM for fluid mechanics or elastostatics \cite{Boolakee2023,Boolakee2023a}, the design of second-order accurate boundary formulations on arbitrary domains turns out to be challenging. The main reasons include
\begin{itemize}
    \item mismatch of the number of physical boundary conditions present and the number of missing incoming populations that need to be reconstructed \cite{Maier1996,Zou1997};
    \item the component-wise resolution of vector or tensor-valued quantities in LBM making rotation into a boundary-aligned reference system challenging \cite{Huang2016,Boolakee2023a};
    \item compensation of higher-order derivatives that naturally appear during the asymptotic expansion of boundary formulations \cite{Guo2002,Boolakee2023a}
    \item the intention of preserving a local formulation that requires only minimal neighbor node information \cite{Junk2005BC,Yang2007};
    \item stability issues that only in few specialized cases have been analyzed analytically (such as in this work for aligned boundaries and in \cite{Junk2005BC,Rheinlander2010,Zhao2020bc}).
\end{itemize}
Accordingly, the derivation of boundary formulations on arbitrary geometries as well as Neumann-type boundary conditions will require extensive additional investigations and will be the focus of future studies.

Another limitation of the method is that by coupling the velocity discretization to the spatial discretization the LBM is restricted to Cartesian grids. While this is a  limitation of the LBM irrespective of the equation being solved, there are several methods known from fluid dynamics to overcome it. Among these are interpolation based methods \cite{fakhari2014finite} and semi-Lagrangian off-lattice methods \cite{kramer2017semi} that modify the streaming step. Besides these there exist many established grid refinement methods that couple Cartesian grids with varying resolution \cite{filippova1998grid,dupuis2003theory,lagrava2012advances,geier2009bubble,Kutscher2019}. All these method use interpolation of populations or moments at the interface together with an appropriate reconstruction of the non-equilibrium populations. The so-called compact interpolation methods \cite{geier2009bubble,Kutscher2019} acquire the interpolation polynomial by including the gradient information from asymptotic analysis. A similar approach would also be applicable for our vectorial LBM, which however requires further detailed investigations. 


\section*{Declaration of competing interest}
The authors declare that they have no known competing financial interests or personal relationships that could have appeared to influence the work reported in this paper.

\section*{CRediT authorship contribution statement}
\textbf{Oliver Boolakee}: Conceptualization, Methodology, Software, Formal analysis, Investigation, Data curation, Writing – original draft, Writing – Review \& Editing, Visualization. \textbf{Martin Geier}: Methodology, Formal analysis, Investigation, Writing – review \& editing. \textbf{Laura De Lorenzis}: Conceptualization, Methodology, Writing – review \& editing, Supervision, Funding acquisition.

\section*{Acknowledgments}
Martin Geier acknowledges financial support by the German Research Foundation (DFG) project number 414265976-TRR 277.

\appendix

\section{From Boltzmann equation to LBM} \label{app:lbm}

In this appendix, we provide a brief summary of the simplification and approximation steps that lead from the Boltzmann equation known from gas kinetics to the lattice Boltzmann equation used in LBM. Note, however, that the derivation of the novel formulation in this work is completely agnostic to these origins of LBM and treats the lattice Boltzmann equation simply as a template for a new numerical method that inherits the beneficial properties of LBM. The primary purpose of this appendix is to provide some more background information on where the lattice Boltzmann equation comes from -- more details can be found in e.\,g. \cite{He1997,Succi2001,Kruger2017,Lallemand2021}.

The starting point is the Boltzmann equation
\begin{equation}
    \partial_t f + \boldsymbol{\xi}\cdot \nabla f = \mathcal{A}(f),
\end{equation}
where $f = f(\boldsymbol{x},t,\boldsymbol{\xi})$ is the distribution function that depends on position $\boldsymbol{x}\in\mathbb{R}^d$, time $t\in\mathbb{R}$ and microscopic velocity $\boldsymbol{\xi}\in\mathbb{R}^d$, with $d$ being the number of dimensions. The local collision operator is denoted by $\mathcal{A}$, which originally refers to the Boltzmann operator. However, \cite{BGK} proposed the greatly simplified BGK collision operator that instead models a relaxation-type process towards some local equilibrium state $f^{eq}(f)$:
\begin{equation}
    \mathcal{A}(f) = -\frac{1}{\tau}(f-f^{eq}(f)).
\end{equation}
The speed of this relaxation is governed by the relaxation time $\tau$. The connection between this mesoscopic description and the macroscopic physics is achieved through moments, which are weighted integrals over velocity space:
\begin{equation}
    \int_{\mathbb{R}^d}\eta(\boldsymbol{\xi})fd\boldsymbol{\xi},
\end{equation}
with $\eta(\boldsymbol{\xi})$ as some polynomial of the microscopic velocity $\boldsymbol{\xi}$. For example, in order to obtain the so-called zeroth-order or first-order moment, we choose $\eta(\boldsymbol{\xi})=1$ or $\eta(\boldsymbol{\xi})=\boldsymbol{\xi}$, respectively.
Furthermore, some of the moments are conserved during collision; e.g. from conservation of the zeroth-order moment follows
\begin{equation}
\begin{split}
    \int_{\mathbb{R}^d}fd\boldsymbol{\xi} &\overset{!}{=} \int_{\mathbb{R}^d}f^{eq}d\boldsymbol{\xi}.
\end{split}
\end{equation}

As the first simplification step, we replace the dense velocity space $\boldsymbol{\xi}\in\mathbb{R}^d$ with a finite set of velocities $\mathbb{V}^*\subset\mathbb{R}^d$ based on a numerical quadrature of the moment calculation. For the example of the zeroth-order moment, we obtain
\begin{equation}
    \int_{\mathbb{R}^d}fd\boldsymbol{\xi} = \int_{\mathbb{R}^d}f^{eq}d\boldsymbol{\xi} = \sum_{i=1}^N w_{i} f(\boldsymbol{x},t,\boldsymbol{\xi}_i) \qquad \text{for some } w_i=\mathbb{R},\; \boldsymbol{\xi}_i\in\mathbb{V}^* \quad \forall i=1,\dots,N,
\end{equation}
where $N$ is the number of weights $w_i$ and abscissae $\boldsymbol{\xi}_i$. These  need to be carefully chosen -- based on the functional form of $f^{eq}$. Because we require only a finite number of moments to solve physics on the macroscopic scale, it is sufficient if only a finite number of moments is correctly recovered. Hereby, the complexity of the physics on the macroscopic scale governs the size of the velocity set. As a result, we obtain the discrete velocity Boltzmann equation
\begin{equation}
    \partial_t f_i + \boldsymbol{\xi}_i\cdot \nabla f_i = -\frac{1}{\tau}(f_i-f_i^{eq}) + \varphi_i(B) \qquad \forall i=1,\dots,Q,
\end{equation}
which involves the so-called populations $f_i:=w_{i} f(\boldsymbol{x},t,\boldsymbol{\xi}_i)$. Here $Q\geq N$ is the total number of populations and velocities. Additionally, we add a forcing term $\varphi_i(B)$ in order to incorporate some external forcing $B$.

As a next step, we discretize space and time using the method of characteristics along $\boldsymbol{\xi}_i$ for each equation. Integrating from current time $t$ until $t+\Delta t$ along the characteristic yields
\begin{equation}
    f_i(\boldsymbol{x}+\boldsymbol{\xi}_i \Delta t,t+\Delta t) - f_i(\boldsymbol{x},t) = -\int_0^{\Delta t} \left[\frac{1}{\tau}\left(f_i(\boldsymbol{x}+\boldsymbol{\xi}_i \zeta,t+\zeta)-f_i^{eq}(\boldsymbol{x}+\boldsymbol{\xi}_i \zeta,t+\zeta)\right) + \varphi_i(B)\right] d\zeta.
\end{equation}
Applying the second-order trapezoidal rule to approximate the right-hand side integral and with the change of variables $\bar{f}_i:=f_i-\frac{\Delta t}{2}\left(-\frac{1}{\tau}(f_i-f_i^{eq}) + \varphi_i(B)\right)$ \cite{He1998m}, we obtain after some algebra
\begin{equation}
    \bar{f}_i(\boldsymbol{x}+\boldsymbol{\xi}_i \Delta t,t+\Delta t) - \bar{f}_i(\boldsymbol{x},t) = -\frac{\Delta t}{\bar{\tau}}\left(\bar{f}_i(\boldsymbol{x},t)-f^{eq}_i(\boldsymbol{x},t)\right) + \frac{\Delta t}{2}\left(2-\frac{\Delta t}{\bar{\tau}}\right)\varphi_i(B),
\end{equation}
where $\bar{\tau}:=\tau+\Delta t/2$. In order to apply the forcing on the zeroth-order moment (as is the case for our formulation), it turns out that $\varphi_i(B)=w_i B$. Defining $W_i:= w_i/2$, $\omega :=\Delta t/\bar{\tau}$ and dropping the over-bar of $f_i$ leads to
\begin{equation}
    f_i(\boldsymbol{x}+\boldsymbol{\xi}_i \Delta t,t+\Delta t) - f_i(\boldsymbol{x},t) = -\omega\left(f_i(\boldsymbol{x},t)-f^{eq}_i(\boldsymbol{x},t)\right) + \Delta t(2-\omega)W_i B
\end{equation}
Finally, we recover Eq. \eqref{eq:lbe} if we assume $d=2$, switch to two-index notation ($i\rightarrow(i,j)$), re-scale space and/or time so that $\boldsymbol{\xi}_{ij}\in\mathbb{V}^*\rightarrow\boldsymbol{c}_{ij}\in\mathbb{V}$ with $\boldsymbol{c}_{ij}$ as defined in Eq. \eqref{eq:set_vel} and lastly introduce vector-valued populations and forcing.

As a final note, we point out that all formulations proposed in Section \ref{sec:lbm} are established independently of this derivation and only take the algorithmic structure of LBM in the form of Eq. \eqref{eq:lbe} as a starting point.



\bibliographystyle{elsarticle-num}
\bibliography{bibliography}







\end{document}